\documentclass[11pt]{article}
\usepackage{amsfonts}
\usepackage{xr}
\usepackage{graphicx,psfrag}
\usepackage{amsmath}
\usepackage{color,soul,textcomp}
\usepackage{enumerate}
\usepackage{geometry}
\usepackage{amssymb,amsthm}
\usepackage{enumerate}
\usepackage{hyperref}
\usepackage[active]{srcltx}
\numberwithin{equation}{section}
\usepackage[T1]{fontenc}
\usepackage{color}
\newtheorem{theo}{Theorem}[section]
\newtheorem{lem}[theo]{Lemma}

\newtheorem{prop}[theo]{Proposition}

\newcommand{\re}{\Re\mathrm{e}} 
\newcommand{\im}{\Im\mathrm{m}} 
\newcommand{\eps}{\varepsilon}

\newcommand{\R}{\mathbb{R}}

\DeclareMathOperator{\cotan}{cotan}
\begin{document}
\title{A KPP road-field system with spatially periodic exchange terms}
\author{Thomas Giletti\footnote{IECL, Universit\'{e} de Lorraine, B.P. 70239, 54506
Vandoeuvre-l\`{e}s-Nancy Cedex, France. E-mail:
thomas.giletti@univ-lorraine.fr.} , L\'{e}onard Monsaingeon\footnote{CAMGSD-IST, University of Lisbon, Av. Rovisco Pais
1049-001 Lisboa, Portugal. E-mail: leonard.monsaingeon@ist.utl.pt.} \ and Maolin Zhou\footnote{School of Science and Technology, University of New England, Armidale, NSW 2351, Australia. E-mail: zhouutokyo@gmail.com.}}
\maketitle

\begin{abstract}
We take interest in a reaction-diffusion system which has been recently proposed~\cite{BRR_influence_of_a_line} as a model for the effect of a road on propagation phenomena arising in epidemiology and ecology. This system consists in coupling a classical Fisher-KPP equation in a half-plane with a line with fast diffusion accounting for a straight road. The effect of the line on spreading properties of solutions (with compactly supported initial data) was investigated in a series of works starting from~\cite{BRR_influence_of_a_line}. We recover these earlier results in a more general spatially periodic framework by exhibiting a threshold for road diffusion above which the propagation is driven by the road and the global speed is accelerated. We also discuss further applications of our approach, which will rely on the construction of a suitable generalized principal eigenvalue, and investigate in particular the spreading of solutions with exponentially decaying initial data.
\end{abstract}
\smallskip 
\hspace{8mm} \small \textbf{Keywords.} Reaction-diffusion systems, asymptotic spreading speed, generalized principal eigenvalues.

\normalsize

\section{Introduction}

The study of the large time behavior of solutions of reaction-diffusion equations has been motivated by a wide range of applications, in particular in ecology and epidemiology. Indeed, it is well understood that propagation phenomena arise from the combining effect of diffusion (which accounts in an ecological context for the motion of individuals) and reaction (reproduction of the species), thus providing a mathematical model for ecological invasions or epidemics~\cite{ShiKa}.

The most classical example is the so-called Fisher-KPP equation
\begin{equation}\label{eq:KPP}
\partial_t v = d \Delta v + f(v), \quad t>0 , \ x \in \R^N , \end{equation}
where $d>0$ and the nonlinearity $f \in C^{1,r}_{loc} ([0,\infty))$, $0<r<1$, satisfies
\begin{equation}\label{hyp:KPP}
f(0)=f(1)=0, \; v \mapsto \frac{f(v)}{v} \mbox{ is decreasing,} \; \mbox{ and } \frac{f(v)}{v} \rightarrow -\infty \mbox{ as } v \rightarrow +\infty .
\end{equation}
The typical example is $f(v) = v (1-v)$ and was investigated in the seminal papers~\cite{Fisher,KPP}. In particular, this equation is well-known to admit positive traveling waves, namely planar solutions moving with a constant shape and speed through the domain and connecting the two constant steady states 0 and 1~\cite{AW75,Fisher,KPP}. The minimal wave speed, which can be explicitly computed as 
$$c^*_{KPP} = 2 \sqrt{d f'(0)},$$
is also the asymptotic spreading speed of solutions of the Cauchy problem~\eqref{eq:KPP} with compactly supported initial data~\cite{AW78}. The striking feature of these results is that, under the KPP assumption~\eqref{hyp:KPP}, the speed $c^*_{KPP}$ is fully determined by the linearization of~\eqref{eq:KPP} around $v=0$.

Various extensions of the KPP equation have been investigated in the past years. In particular, a tremendous effort has been dedicated to the inclusion of heterogeneity, which is a common yet mathematically challenging feature of more realistic ecological models~\cite{ShiKa}. Among the vast literature, we refer in particular to~\cite{BH-front,BHNn08,BHNi-I,BHH-periodic_fragmented,Wein02} for generalizations of the results mentioned above in the spatially periodic framework, in which we will also place ourselves below. Let us also mention that further heterogeneous problems, as well as other types of nonlinearities have also been studied: we refer again to~\cite{BH-front} and the references therein, as well as to the thorough reviews~\cite{Ber02-adv,xin00}.\\

More recently, Berestycki, Roquejoffre and Rossi proposed a new model specifically devised for studying the role of roads in biological invasions~\cite{BRR_influence_of_a_line}. This issue was motivated by the empirical observation that not only diseases, but also various ecological species, tend to spread faster along roads and other transport lines. 

Their model is a system of two parabolic equations, whose unknowns $u(t,x)$ and $v(t,x,y)$ should be interpreted as the density of a same population but on different domains, respectively the road and the field:
\begin{equation}
\left\{
\begin{array}{ll}
\partial_t u-D\partial_x^2 u=\nu(x)v(t,x,0)-\mu(x) u, & t>0,x\in\R ,\\
\partial_t v-d\Delta v=f(v), & t>0,(x,y)\in\Omega ,\\
-d\partial_y v(t,x,0)=\mu (x) u- \nu (x) v(t,x,0), & t>0,x\in\R .
\end{array}
\right. 
\label{eq:model}
\end{equation}
From now on and consistently with this interpretation, we will refer to the upper half plane $\Omega = \{(x,y) \in \R^2 \, : \ y >0\}$ as the field, and to the line $\{ (x,0) \, : \ x \in \R\}$ as the road. 

In~\cite{BRR_influence_of_a_line} as well as throughout this work, the field equation (second line of \eqref{eq:model}) is still assumed to be of the KPP type in the sense that the function $f$ satisfies~\eqref{hyp:KPP}. However, it is coupled with the road equation (first line of \eqref{eq:model}) through a Robin type boundary condition at $y=0$ (last line of \eqref{eq:model}). Roughly speaking, $\mu u$ is the proportion of individuals jumping from the road to the field, and $\nu v$ the proportion of individuals jumping from the field to the road. Note that the field boundary condition is normalized so that, when no reproduction takes place ($f \equiv 0$), then the total mass of the population is conserved (we refer to~\cite{BRR_influence_of_a_line} for a complete argument).

When $\mu$ and $\nu$ are positive constants, Berestycki, Roquejoffre and Rossi~\cite{BRR_influence_of_a_line} proved that the solution (with compactly supported initial data) spreads along the road with some speed $c^* (D) \geq c^*_{KPP} = 2\sqrt{d f'(0)}$ and that, moreover, this inequality is strict if and only if $D >2d$. This in particular means that, to accelerate the propagation, fast diffusion on the road is enough even though individuals do not reproduce there. One may also note that this diffusion threshold $D =2d$ for acceleration does not depend on the values of the constants $\mu$ and $\nu$, an observation which will be strengthened by our results.

A more general formula for the diffusion threshold, including the case when transport or reproduction take place on the road, was given in~\cite{BRR_plus}: in particular, when reproduction rates are identical on the road and in the field, then acceleration occurs as soon as $D>d$. Related results were also obtained in various frameworks such as fractional diffusion~\cite{BCoulonRR}, ignition type nonlinearities~\cite{Dietrich1,Dietrich2} and nonlocal exchange~\cite{Pauthier}. The issue of the spreading speed in oblique directions away from the road was investigated in~\cite{BCoulonRR,BRR_shape_of_expansion} and revealed interesting properties where the shape of the propagation is no longer dictated by a Huygens principle, and may not even be convex. Lastly, we refer to~\cite{Tellini} where the restriction of the KPP problem~\eqref{eq:model} to a truncated field, with a Dirichlet boundary condition and constant exchange coefficients, was studied. It turns out that such a truncation plays an essential role in 
our arguments.

While all the above papers dealt with systems that are invariant in the $x$-direction (let us point out here that the presence of the road itself entails strong heterogeneity with respect to the $y$-direction), our work will stand in a more general spatially periodic framework where the results of~\cite{BRR_influence_of_a_line} (as well as some of~\cite{BRR_plus,Tellini}) will be extended. Even though the heterogeneity will be limited to the exchange terms, we believe our arguments could be extended to more general periodic systems and will briefly discuss this in the last section.

\subsection{Main results}

From now on, we consider the system~\eqref{eq:model} where $D$, $d>0$, the function $f$ satisfies the KPP assumption~\eqref{hyp:KPP}, and $\mu(x),\nu(x)$ are $L$-periodic exchange coefficients in $C^{1,r} (\R)$ such that
$$\mu \geq \not \equiv 0 , \quad \nu \geq \not \equiv 0.$$
As mentioned above the main difference with the original model of~\cite{BRR_influence_of_a_line} lies in the $x$-dependence of the exchange terms, which is motivated by the ubiquity of heterogeneity in the applications. The periodicity provides a suitable mathematical framework, while still requiring a new approach avoiding explicit computations. We also point out that $\mu$ and $\nu$ may occasionally be 0, which of course could not happen in the homogeneous exchange case. This can be interpreted as the presence of "walls" blocking locally the motion of individuals between the field and the road (but not necessarily both ways simultaneously).

For convenience, we introduce the notations
$$
\mu_0 := \min \mu \leq \mu(x)\leq \mu_1:= \max \mu, \quad  \nu_0:= \min \nu \leq \nu(x)\leq \nu_1 := \max \nu.
$$
The system~\eqref{eq:model} is supplemented with bounded and non-negative initial data
\begin{equation}
\left\{
\begin{array}{ll}
u|_{t=0}=u_0 & \text{in } \R,\\
v|_{t=0}=v_0 & \text{in } \Omega .
\end{array}
\right.
\label{eq:initial_data}
\end{equation}
As will be shown in Section~\ref{sec:cauchy_stationary}, this problem is well-posed under generic conditions on the initial data.\\

In this paper, we take interest in the large time behavior of the solution and more specifically in its spreading speed in the $x$-direction or, in other words, along the road. By spreading, we mean here that the solution should converge to a positive steady state on some expanding in time domains. Therefore, we will begin with a Liouville type result to describe such positive steady states:
\begin{theo}
There exists a unique positive and bounded stationary solution $(U,V)$. This solution is $L$-periodic in $x$, has a positive infimum and satisfies $V(x,+\infty)\equiv 1$ where the limit is uniform with respect to $x \in \R$.
\label{main:liouville}
\end{theo}
Unlike in the homogeneous exchange case, it is not possible to explicitly compute this positive steady state and, therefore, both its existence and uniqueness will be inferred from the (linear) instability of 0 and the KPP assumption. The proof, which will be performed in Section~\ref{sec:liouville42}, relies on the fact that the road-field system satisfies a comparison principle, namely Propositions~\ref{prop:comparison1} and~\ref{prop:comparison2}. Indeed, although it is non-standard due to the 1D-2D coupling through a boundary condition, this model shares many features with monotone systems. We refer the reader to~\cite{LiangZhao} and the references therein for a study of abstract monotone systems.

From the proof of Theorem~\ref{main:liouville} it will be clear that this steady state is stable. One may thus expect it to attract solutions of \eqref{eq:model}, leading to the spreading dynamics that we wish to investigate:
\begin{theo}\label{main:spread1}
There exists $c^* (D) \geq c^*_{KPP} = 2 \sqrt{d f'(0)} >0$ such that any solution $(u,v)$ of \eqref{eq:model}-\eqref{eq:initial_data} with non-negative, continuous and compactly supported initial data $(u_0,v_0) \not \equiv (0,0)$ satisfies:
\begin{itemize}
\item for all $c > c^* (D)$ and $R>0$,
$$\lim_{t \to +\infty}\   \sup_{x \leq -ct\; ,\; 0 \leq y \leq R} ( u (t,x) + v(t,x,y)) = 0;$$
\item for all $0 < c < c^*  (D)$ and $ R >0$,
$$\lim_{t \to +\infty} \  \sup_{0 \geq x \geq -ct\; ,\; 0 \leq y \leq R } \left( | u(t,x) - U (x) | + | v(t,x,y) - V(x,y)| \right) = 0.$$
\end{itemize}
\end{theo}
The above statement deals with spreading in the left direction. As the symmetric problem $x\leftrightarrow -x$ satisfies the same assumptions, there clearly also exists a spreading speed in the right direction. However, the left and right spreading speeds may be different. 

The proof is directly inspired by the single (spatially periodic) equation. It relies on the construction of a family of principal eigenvalues denoted by $\Lambda (\alpha)$, which arises when looking for exponential solutions
$$
e^{\alpha (x+ct)} (U,V)
$$
of the linearized problem around the invaded unstable state $(u=0,v=0)$, where $\alpha >0$ and $U (x)$, $V(x,y)$ are positive and $L$-periodic with respect to $x$. Due to the unboundedness of the domain, the usual notion of principal eigenvalue is not uniquely defined and there exist several definitions of "generalized" principal eigenvalues~\cite{BR-general_eigen}. Our generalized principal eigenvalue $\Lambda (\alpha)$ (see Theorem~\ref{th:eigenvalue} for its main properties) will be constructed in Section~\ref{sec:linear_pb} by truncating the field in the $y$-direction and passing to the limit to the whole road-field domain.

Then in Section~\ref{section:spreading}, we will prove Theorem~\ref{main:spread1} and show that the spreading speed $c^* (D)$ is characterized by the following formula:
$$c^* (D) = \min_{\alpha >0} \frac{-\Lambda (\alpha)}{\alpha}.$$
This is the natural extension of a similar formula for the minimal wave speed of the single spatially periodic equation~\cite{BH-front,BHNi-I}, which in the homogeneous case \eqref{eq:KPP} easily reduces to
\begin{equation}\label{eq:almost}
c^*_{KPP} = \min_{\alpha >0} \frac{d\alpha^2 + f'(0)}{\alpha}.
\end{equation}

Therefore, whether the road accelerates the propagation or not depends on how $-\Lambda (\alpha)$ compares to $d\alpha^2 + f'(0)$ (note that it is already stated in Theorem~\ref{main:spread1} that $c^* (D) \geq c^*_{KPP}$). In fact we will prove that $-\Lambda (\alpha)$ is strictly larger than $d\alpha^2 + f'(0)$ if and only if $D>d$ and $\alpha > \sqrt{\frac{f'(0)}{D-d}}$, see Proposition~\ref{prop:eigen_enhance}. Recalling that $\alpha$ denotes the exponential decay of the ansatz as $x \to - \infty$, the formal argument reads as follows: when looking at the linearized problem \eqref{eq:principal_eigenvalue} later on, the intrinsic growth rate of the road population is $D \alpha^2$ (due to motion of road individuals in the original problem) and, in the field, it is $d\alpha^2 + f'(0)$ (due to both the motion of individuals in the field and their reproduction). Thus, one may indeed expect the road to lead the propagation if and only if $\alpha > \sqrt{\frac{f'(0)}{D-d}}$. Notice then that, as the 
exponential 
decay of the KPP traveling wave with minimal speed is $\sqrt{\frac{f'(0)}{d}}$, which is also where the minimum in \eqref{eq:almost} is reached, the diffusion threshold $D=2d$ immediately arises.

The above argument will be made rigorous in Section~\ref{sec:accelerate}, leading to the following theorem:
\begin{theo}\label{main:speedup}
The strict inequality $c^* (D) > c^*_{KPP}$ holds if and only if $D >2d$.

Furthermore,$$0 < \liminf_{D \to \infty} \frac{ c^*(D)}{\sqrt{D}} \leq \limsup_{D \to \infty} \frac{ c^*(D)}{\sqrt{D}} < +\infty.$$
\end{theo}
We highlight the fact that we recover the exact same diffusion threshold $D=2d$ as in~\cite{BRR_influence_of_a_line}, not only in spite of the heterogeneous exchange, but also in spite of the fact that $\mu$ and $\nu$ may occasionally be 0.\\

The above discussion also leads one to expect that if the initial data are slowly exponentially decaying as $x \to -\infty$, then the solution may no longer be accelerated by the road. This can be easily interpreted since the effect of diffusion gets weaker as the solution is relatively flat. More generally we will prove in Section~\ref{sec:exp_case} that, as in the single equation case, the family of eigenvalues $-\Lambda (\alpha)$ also characterizes the spreading speed of solutions of \eqref{eq:model} with any exponential decay:

\begin{prop}\label{exp_decay_case}
Let $\alpha^*>0$ be the smallest solution of $c^* (D) \alpha = -\Lambda (\alpha)$, and $(u,v)$ be the solution of \eqref{eq:model}-\eqref{eq:initial_data} with non-negative, bounded and continuous initial data $(u_0,v_0)$ satisfying also
$$
\forall \,x\leq 0:\qquad \left\{
\begin{array}{l}
m(0) e^{\alpha x} \leq u_0 (x) \leq M e^{\alpha x} ,\\
m(y) e^{\alpha x} \leq v_0 (x,y) \leq M e^{\alpha x},
\end{array}
\right.
$$
for some $0 < \alpha < \alpha^*$, $M>0$, and $m$ a continuous and positive function of $y \in [0,\infty)$.

Then $(u,v)$ spreads along the road (in the same sense as in Theorem~\ref{main:spread1}) with speed
$$c (\alpha) := \frac{-\Lambda (\alpha)}{\alpha}.$$
In particular, if either $D \leq d$ or $D>d$ and $\alpha \leq \sqrt{\frac{f'(0)}{D-d}}$, then $c (\alpha) = d \alpha + \frac{f'(0)}{\alpha}$ is also the speed of solutions of~\eqref{eq:KPP} with exponential decay of order $e^{\alpha x}$.
\end{prop}
Combining Theorem~\ref{main:spread1} and Proposition~\ref{exp_decay_case}, it is clear by a simple comparison argument that any solution which decays at least as $e^{\alpha^* x}$ as $x \to - \infty$ at time $t=0$ spreads with the speed $c^* (D)$. In particular, it turns out that the behavior of the initial data in the $y$-direction barely matters, which is in some sense natural but also not trivial as the eigenfunctions associated with $-\Lambda (\alpha)$ may decay as $y \to +\infty$.

\paragraph{Outline of the paper.} The paper is organized as follows: in Section~\ref{sec:cauchy_stationary} we investigate well-posedness of the Cauchy problem~\eqref{eq:model} and its stationary solutions. In Section~\ref{sec:linear_pb} we construct the principal eigenvalue of the linearized problem and the associated exponential solutions. Section~\ref{section:spreading} contains the proof of the main spreading Theorem~\ref{main:spread1}. We investigate in Section~\ref{sec:accelerate} the acceleration properties and the limit as $D\to\infty$ of large diffusion on the road, that is Theorem~\ref{main:speedup}, as well as the spreading properties for exponentially decaying initial data. 

Lastly we will discuss in Section~\ref{sec:extension} some possible extensions of our work. We will focus in particular on the truncated field case, i.e. the restriction of \eqref{eq:model} to the domain $\{ 0 \leq y < R\}$, and on the case when reproduction also takes place on the road. We will see that most of our results can be extended in both those frameworks with only minor modifications. We will then discuss the general spatially periodic problem and point out some remaining open questions.

\section{The Cauchy problem and stationary solutions}\label{sec:cauchy_stationary}

\subsection{Preliminaries: the evolution problem}

We begin by noting that the parabolic system satisfies the comparison principle. From now on, we say that $(u,v)$ is a subsolution (respectively supersolution) of \eqref{eq:model} if it is satisfied with inequalities $\leq$ (respectively $\geq$) instead of equalities, and if both functions $u$ and $v$ are continuous.

\begin{prop}\label{prop:comparison1}
Let $(\underline{u},\underline{v})$ and $(\overline{u},\overline{v})$ be respectively a subsolution bounded from above and a supersolution bounded from below such that $(\underline{u},\underline{v})\leq(\overline{u},\overline{v})$ at time $t=0$. 

Then, either $(\underline{u},\underline{v})<(\overline{u},\overline{v})$ for all $t>0$, or there exists $T>0$ such that $(\underline{u},\underline{v})\equiv(\overline{u},\overline{v})$ for all $t\leq T$. 
\end{prop}
\begin{proof}
The proof is the same as in~\cite{BRR_influence_of_a_line} and we omit the details. The argument relies on a combination of the strong parabolic maximum principle and parabolic Hopf lemma, together with the key assumption that $\mu$ and $\nu$ are non-negative.\end{proof}

This comparison principle also extends to generalized sub and supersolutions. We again omit the proof and refer to Proposition~3.3 in~\cite{BRR_influence_of_a_line} for the details. We write their result below for the convenience of the reader and because we will use it several times throughout this paper:
\begin{prop}\label{prop:comparison2}
Let $E \subset (0,+\infty) \times  \R $ and $F \subset (0,+\infty) \times \Omega $ be two open sets. Let $(u_1,v_1)$ and $(u_2,v_2)$ be two subsolutions of \eqref{eq:model}, bounded from above and satisfying
$$u_1 \leq u_2 \ \mbox{ on } \partial E \cap ((0,+\infty) \times \R ),\quad v_1 \leq v_2 \ \mbox{ on } \partial F \cap ((0,+\infty) \times  \Omega ),$$
and define the functions
$$\underline{u} (t,x) := \left\{
\begin{array}{ll}
\max \{ u_1 ,u_2 \} & \mbox{ in } \overline{E},\\
u_2 & \mbox{ otherwise},
\end{array}
\right.$$
$$\underline{v} (t,x) := \left\{
\begin{array}{ll}
\max \{ v_1 , v_2 \} & \mbox{ in } \overline{F},\\
v_2  & \mbox{ otherwise}.
\end{array}
\right.$$
If it satisfies
\begin{eqnarray*}
&&  \underline{u}(t,x) > u_2 (t,x) \Rightarrow \underline{v} (t,x,0) \geq v_1 (t,x,0), \\
&& \underline{v} (t,x,0) > v_2 (t,x,0) \Rightarrow \underline{u} (t,x) \geq u_1 (t,x),
\end{eqnarray*}
then, any supersolution $(\overline{u}, \overline{v})$ of \eqref{eq:model} such that $\underline{u}\leq \overline{u} $ and $\underline{v} \leq \overline{v} $ at time $t=0$, also satisfies $\underline{u} \leq \overline{u}$ and $\underline{v} \leq \overline{v}$ for all time $t>0$.
\end{prop}
Thanks to those comparison principles, we can prove as announced that the Cauchy problem is well-posed.
\begin{theo}\label{thm:cauchy}
The Cauchy problem \eqref{eq:model}-\eqref{eq:initial_data}, where the initial datum is bounded, non-negative and H\"older continuous, has a unique global non-negative and bounded solution.
\end{theo}
\begin{proof}
Note that uniqueness immediately follows from the previous proposition, so it only remains to prove the existence part. The proof is again very similar to~\cite{BRR_influence_of_a_line}, by constructing in the spirit of \cite{Sattinger-Monotone} a monotonic sequence of solutions $(u^n,v^n)_{n \geq 1}$ to
\begin{equation}
\label{eq:iterate_u}
\left\{
\begin{array}{ll}
\partial_t u^n -D\partial_x^2 u^n + \mu (x) u^n = \nu(x)v^{n-1} (t,x,0), & t>0,x\in\R, \vspace{3pt}\\
u^n|_{t=0} = u_0, & x \in \R,
\end{array}
\right.
\end{equation}
and
\begin{equation}
\label{eq:iterate_v}
\left\{
\begin{array}{ll}
\partial_t v^n -d\Delta v^n=f(v^n),  & t>0,(x,y)\in\Omega, \vspace{3pt} \\
-d\partial_y v^n(t,x,0)=\mu(x) u^n- \nu (x)v^n (t,x,0), & t>0,x\in\R , \vspace{3pt}\\
v^n|_{t=0} = v_0 , & x \in \R,
\end{array}
\right. 
\end{equation}
where $v^0 \geq 0$ will be suitably initialized when $n=1$ in \eqref{eq:iterate_u} later on.

By classical parabolic theory \cite{LadyzhenskayaUraltseva} both problems shall be solvable at each step, and the minimum principle ensures that $u^n,v^n$ stay non-negative. In order to pass to the limit and obtain a solution of \eqref{eq:model}, we need some $L^\infty$ estimates on $(u^n,v^n)$ uniformly with respect to~$n$. This is done by exhibiting a bounded supersolution (which will also serve in the initialization of the sequence) and applying the comparison principle.

First let $\tilde{U}_0(x)$ be the (positive) principal eigenfunction of the periodic problem
\begin{equation*}
\left\{
\begin{array}{ll}
 -D \tilde{U}_0 ''+ \mu(x) \tilde{U}_0 = \lambda_0 \tilde{U}_0 , & x\in\R, \vspace{3pt}\\
\tilde{U}_0 \mbox{ is $L$-periodic}.
\end{array}
\right.
\end{equation*}
Note that from our assumptions on $\mu$, it is clear that this principal eigenvalue $\lambda_0 >0$.

Now let
\begin{equation}
\label{supersol442}
(\overline{U}(x),\overline{V}(x,y)) :=  K \left(\tilde{U}_0(x),C(1+e^{-\omega y})  \right)
\end{equation}
for some parameters $K,C,\omega>0$ to be adjusted shortly. Fix first $C\leq \frac{\lambda_0\min\tilde{U}_0}{2\nu_1}$ and then $\omega\geq \frac{\mu_1\max\tilde{U}_0}{Cd}$. By the KPP hypothesis \eqref{hyp:KPP} we can next choose $K>0$ large enough such that $\frac{f(KC)}{KC}\leq -d\omega^2$, and some straightforward computations show that $(\overline{U},\overline{V})$ satisfies
$$
\left\{
\begin{array}{l}
-DU''+\mu(x)U\geq \nu(x) V(x,0) , \\
 -d\Delta V\geq f(V), \\
 -d\partial_y V(x,0)+\nu(x)V(x,0)\geq \mu(x) U.
\end{array}
\right.
$$
Note that this positive stationary supersolution is bounded from above and away from zero: up to increasing $K>0$ (for fixed $C,\omega$) we can therefore assume that the initial data $0\leq (u_0,v_0)\leq (\overline{U},\overline{V})$.

Initializing $v^0(t,x)=\overline{V}(x)$ in \eqref{eq:iterate_u}, a straightforward inductive application of the comparison principle,  Proposition \ref{prop:comparison1}, shows that
$$
\overline{U}\geq \ldots \geq u^{n}\geq u^{n+1}\geq \ldots\geq 0
\quad\mbox{and}\quad
\overline{V}\geq \ldots \geq v^{n}\geq v^{n+1}\geq \ldots\geq 0
$$
for all $n\geq 1$, and therefore
$$
\forall t\geq 0,\,x\in\R,\,y\geq 0:\qquad u^n(t,x,y)\searrow u(t,x,y)
\quad
\mbox{and}\quad 
v^n(t,x)\searrow v(t,x)
$$
as $n\to\infty$ pointwise.

In order to show that the pair $(u,v)$ satisfies \eqref{eq:model} together with the initial condition \eqref{eq:initial_data}, we claim now that the sequence $\{u^n,v^n\}_{n\geq 1}$ is relatively compact in the local uniform, the $C^{0,1}_{loc} ((0,\infty)\times \overline{\Omega})$, and the $C^{1,2}_{loc} ((0,\infty) \times \Omega)$ topologies. More precisely, for any $R>0$ and $T>0$, the uniform boundedness of $v^n$ implies uniform $W^{1,2}_{p} ((0,T) \times B_R)$ estimates on the sequence $u^n$ for any large $p$, and in particular it is locally uniformly H\"{o}lder continuous by Morrey's inequality. Then by parabolic estimates~\cite{LadyzhenskayaUraltseva} on the field equation the sequence $v^n$ is locally uniformly H\"{o}lder continuous, as well as its spatial derivatives of order 1 for all $t>0$ and up to the road boundary $\{y=0\}$. Furthermore, by standard interior Schauder estimates, its time derivative and space derivatives up to the order~2 are locally uniformly H\"{o}lder continuous inside the field. This is enough to obtained the desired compactness.

Clearly any cluster point must agree with the previous pointwise limit $(u^n,v^n)\searrow (u,v)$, so by standard uniqueness and separation arguments we conclude that the whole sequence converges in the strong topologies. This is enough to pass to the limit in \eqref{eq:iterate_u}-\eqref{eq:iterate_v} and retrieve \eqref{eq:model} for all positive times. Moreover, $(u,v)=\lim (u^n,v^n)$ is continuous up to time $t=0$, and the proof is achieved.
\end{proof}

\subsection{The Liouville type result}\label{sec:liouville42}

We look now for stationary solutions, i.e. solutions $(U,V)$ of
\begin{equation}
\left\{
\begin{array}{ll}
-D\Delta U=\nu(x)V(x,0)-\mu(x) U , & x\in\R,\vspace{3pt}\\
-d\Delta V=f(V), & (x,y)\in\Omega ,\vspace{3pt}\\
-d\partial_y V(x,0)=\mu(x) U(x)-\nu(x)V(x,0) , & x\in\R,
\end{array}
\right. 
\label{eq:model_stationary}
\end{equation}
and prove Theorem~\ref{main:liouville}.
Observe that, in the statement of Theorem~\ref{main:liouville}, the periodicity of $(U,V)$ immediately follows from uniqueness since any (lattice) translation of $(U,V)$ is again a solution. Moreover, the limit $V(x,+\infty) \equiv 1$ is a straightforward consequence of the positivity of the infimum: indeed, up to extraction of a subsequence, $V$ converges as $y \to +\infty$ to a positive and bounded stationary solution of $-d \Delta V = f(V)$, which may only be 1 by the KPP assumption.\\

Let us now begin the proof of Theorem~\ref{main:liouville} with the existence part:
\begin{lem}
There exists at least one non-trivial, non-negative and bounded solution of \eqref{eq:model_stationary}.
\label{lem:liouville_exist}
\end{lem}
\begin{proof}
We construct below some non-negative subsolution $(\underline{U},\underline{V})$ of \eqref{eq:model} which is bounded and non-trivial. We use the classical method and first let $R>0$ large enough such that the principal eigenvalue $\lambda(R)$ of $-d\Delta$ in the ball $B_R(0,0)\subset\R^2$ with Dirichlet boundary condition is less than $f'(0)/2$. If $\Phi_R(x,y)\geq 0$ is the corresponding principal eigenfunction, set
$$
\tilde{U}(x):=0,\qquad \tilde{V}(x,y):=\left\{
\begin{array}{ll}
\Phi_R (x,y-R-1) & \text{if }(x,y)\in B_R(0,R+1),\\
0 & \text{elsewhere}.
\end{array}
\right. 
$$
Choosing then a constant $C>0$ small enough and thanks to the regularity of~$f$, $(\underline{U},\underline{V})=C(\tilde{U},\tilde{V})$ will automatically be a subsolution of \eqref{eq:model_stationary}.

By Theorem~\ref{thm:cauchy}, there exists a solution $(u,v)$ of \eqref{eq:model} with initial datum $(\underline{U},\underline{V})$. Furthermore, as \eqref{eq:model} satisfies a comparison principle (see Propositions~\ref{prop:comparison1} and~\ref{prop:comparison2}), it lies above $(\underline{U},\underline{V})$ for all time and, furthermore, it is non-decreasing with respect to time. On the other hand, up to decreasing $C$, we can also assume that $(\underline{U},\underline{V}) \leq (\overline{U},\overline{V})$ where $(\overline{U},\overline{V})$ is the supersolution~\eqref{supersol442} constructed in the previous section. In particular, $(u,v)$ is bounded from above, hence it converges as $t \to +\infty$. Recalling that this supersolution is uniformly bounded, one can use standard parabolic estimates (see also the end of the proof of Theorem~\ref{thm:cauchy}) to get that $(u,v)$ converges locally uniformly to a stationary solution $(U,V)$ of \eqref{eq:model_stationary}. It is clear by construction that $(
U,V)$ is 
non-trivial, non-negative and 
bounded.
\end{proof}

\begin{lem}
Let $(U,V)$ be a non-negative, non-trivial, and bounded solution of \eqref{eq:model_stationary}: then $\inf_{\R}U>0$ and $\inf_{\overline{\Omega}}V>0$.
\label{lem:positivity} 
\end{lem}
\begin{proof}
We show first that $V\geq 0$ is bounded away from zero, and positivity of~$U$ will follow.

If $V\geq 0$, then the classical strong maximum principle, applied in the field, implies $V>0$. Using again the first Dirichlet eigenfunction of $-d\Delta$ in a large ball $B_{R}\subset \Omega$ and moving the ball it is easy to show as in \cite[Proposition 4.1]{BRR_influence_of_a_line} that
\begin{equation}
\forall r>0,\qquad \displaystyle{\inf_{y\geq r}}\; V>0.
\label{eq:inf_V}
\end{equation}
In other words, $V$ is uniformly positive far from the road. Arguing by contradiction, assume that $V(x_n,y_n)\searrow 0$ for some sequence $(x_n,y_n)_{n\geq 0}$ such that $y_n\searrow 0$. Shifting
$$
\begin{array}{ll}
U_n(x):=U(x+x_n), & V_n(x,y):=V(x+x_n,y+y_n), \vspace{3pt}\\
\mu_n(x):=\mu(x+x_n), & \nu_n(x):=\nu(x+x_n),
\end{array}
$$
elliptic estimates~\cite{GilbargTrudinger} and periodicity of $\mu$ and $\nu$ show, up to extraction of a subsequence, that $U_n$, $V_n$, $\mu_n$, and $\nu_n$ converge locally uniformly to some $\tilde{U},\tilde{V},\tilde{\mu},\tilde{\nu}$ satisfying the same system in the same domains, and
$$
\tilde{V}\geq 0\  \text{ in }\overline{\Omega},\qquad \tilde{V}(0,0)=0.
$$
In fact \eqref{eq:inf_V} implies that $\tilde{V}>0$ in $\Omega$. Since $-d\Delta \tilde{V}=f(\tilde{V})\geq 0$ in some neighborhood of $(0,0)$, Hopf lemma finally yields the following contradiction
$$
0>-d\partial_y\tilde{V}(0,0)=\tilde{\mu}(0)\tilde{U}(0)- \tilde{\nu}(0)\tilde{V}(0,0)\geq 0.
$$
Therefore, $\inf_{\overline{\Omega}} V >0$ as announced.

Assume now there exists $\left(x_n\right)_{n\geq 0}$ such that $U(x_n)\searrow 0$. As above, up to translating in $x$ and up to extraction of a subsequence, we may assume that
$$
U(x)\geq 0\text{ in }\R,\qquad U(0)=0.
$$
Here, abusing of notations, we still denote by $U$, $V$, $\mu$ and $\nu$ the shifted functions satisfying the same system. Since $V$ is positive,
we have
$$
-D\Delta U + \mu U = \nu V (\cdot,0) \geq 0.
$$
Hence, by the strong maximum principle, we get that $U \equiv 0$. But then, $\nu(\cdot)V (\cdot,0) \equiv 0$, and choosing any $x_0$ such that $\nu(x_0)>0$ contradicts the fact that $V$ has positive infimum. This completes the proof.
\end{proof}
As mentioned above, it follows from the positive infimum of $V$ and the KPP assumption that
$$
V(\cdot,+\infty) \equiv 1.
$$
As explained at the beginning of this section, it only remains to prove the uniqueness of positive and bounded stationary solutions.
\begin{lem}
There exists at most one non-negative and non-trivial bounded solution to \eqref{eq:model_lin}.
\end{lem}
\begin{proof}
The idea is similar to \cite{BHH-periodic_fragmented} (we refer more specifically to the proof of their Theorem~2.4). The argument strongly relies on the uniform positivity from Lemma~\ref{lem:positivity}, which followed from the instability of $(0,0)$ as a steady state of \eqref{eq:model}. As this instability may be reinterpreted in terms of the generalized principal eigenvalue defined in Section~\ref{sec:linear_pb}, we also refer to the deeply related Liouville type results of~\cite{BHR-Liouville} for the general heterogeneous KPP equation.

Assume $(U_1,V_1)$ and $(U_2,V_2)$ are two such solutions, and let
$$
\theta^*:=\sup\{\theta>0 \ : \quad(U_1,V_1)>\theta(U_2,V_2)\}.
$$
Since both solutions are bounded away from zero (Lemma~\ref{lem:positivity}) and from above, $\theta^*$ is finite and positive. We show below that $\theta^*\geq 1$ and therefore $(U_1,V_1)\geq(U_2,V_2)$. By symmetry we would also retrieve the opposite inequality, hence $(U_1,V_1)\equiv (U_2,V_2)$.
\par
We proceed by contradiction and assume that $\theta^* < 1$. By continuity there holds
$$
P(x):=U_1(x)-\theta^*U_2(x)\geq 0,\qquad Q(x,y):=V_1(x,y)-\theta^*V_2(x,y)\geq 0,
$$
and by definition of $\theta^*$ there exists a sequence $(x_n,y_n)_{n\geq 0}$ such that $P(x_n)\searrow 0$ or $Q(x_n,y_n)\searrow 0$ (or both). 

We first consider the second case. Note that
$$Q (x,y) \to 1 - \theta^*>0
$$
as $ y\to +\infty$, thus we may assume without loss of generality that the sequence~$y_n$ is bounded and converges to some $y_\infty \geq 0$.

Assuming first that $y_\infty>0$, shifting in the $x$ direction as before, and extracting a subsequence we can assume that $Q(0,y_\infty) = 0$ and $Q \geq 0$ satisfies
 $$
 -d\Delta Q=f(V_1)-\theta^*f(V_2)>f(V_1)-f(\theta^*V_2).
 $$
Here we took advantage of the KPP hypothesis ($f(v)/v$ decreasing) and of the inequality $\theta^*<1$. Thus
 \begin{equation}
 -d\Delta Q+aQ>0
 \label{eq:LQ>0}
\end{equation}
in $\R \times (0,+\infty)$, where
$$
a(x,y)=-\frac{f(V_1)-f(\theta^*V_2)}{V_1-\theta^*V_2}
$$
is bounded uniformly in $x,y$ because $V_1,V_2$ are and $f$ is Lipschitz. Since $Q$ attains an interior minimum point $(0,y_{\infty})\in\Omega$, the strong maximum principle implies $Q\equiv 0$, which contradicts the strict inequality in \eqref{eq:LQ>0}.

Assuming now that $y_\infty  = 0$ and shifting again, we may take without loss of generality $Q(0,0) =0$ and, according to the previous argument, $Q>0$ in $\Omega$. Since \eqref{eq:LQ>0} still holds, Hopf lemma shows that $\partial_y Q(0 ,0)>0$, which is impossible since
$$
0>-d\partial_yQ(0 ,0)=\mu(0 )P(0 )-\nu(0 ) Q(0 ,0)=\mu(0 )P(0 )
$$
and $P\geq 0$ by definition of $\theta^*$.

Since we just proved that $Q>0$ up to the boundary we only have to prove that $P>0$, which will yield the desired contradiction. If not, then shifting again, extracting a subsequence, and with the same abuse of notations, we may assume that $P(0)=0$, where $P$ still satisfies the first equation
$$
-D\Delta P+\mu(x)P=\nu(x)Q(x,0),\qquad x\in \R.
$$
Since $Q\geq 0$ and $\mu(x),\nu(x)\geq 0$ the strong maximum principle implies $P(x)\equiv P(0 )=0$. The previous equality immediately implies $\nu(\cdot ) Q(\cdot ,0)\equiv 0$, thus $Q(x_0,0)= 0$ at any point $x_0$ such that $\nu(x_0)>0$. According to the previous argument this is a contradiction, and the proof is complete.
\end{proof}
%
%
\section{A generalized principal eigenvalue}\label{sec:linear_pb}
In this section, we take interest in the linearized problem of~\eqref{eq:model}:
\begin{equation}
\left\{
\begin{array}{ll}
\partial_t u-D\Delta u=\nu(x)v(t,x,0)-\mu(x) u , & t>0,x\in\R,\vspace{3pt}\\
\partial_tv-d\Delta v=f'(0)v,  & t>0,(x,y)\in\Omega,\vspace{3pt}\\
-d\partial_y v(t,x,0)=\mu(x) u-\nu(x)v(t,x,0),  & t>0,x\in\R.
\end{array}
\right. 
\label{eq:model_lin}
\end{equation}
We will seek for exponential solutions of \eqref{eq:model_lin} moving with constant speed, i.e. of the form
\begin{equation}
\left(u,v\right)=e^{\alpha(x+ct)}\left(U(x),V(x,y)\right),\qquad c>0,\;\alpha>0
 \label{eq:exponential_supersolutions}
\end{equation}
for some $x$-periodic and positive $U,V$. By analogy with the single equation, we expect that there exists a critical speed $c^* (D)$ such that these exponential pulsating wave solutions exist if and only if $c\geq c^* (D)$. Furthermore, we expect this critical speed $c^* (D)$ to be also the spreading speed of solutions of~\eqref{eq:model} with compactly supported initial data.\\

Fixing a parameter $\alpha \geq 0$, we will denote for convenience
\begin{equation*}
\begin{array}{l}
L_{1} (U,V) := -DU'' - 2D\alpha U' + (-D\alpha^2 + \mu(x))U - \nu(x) V(x,0), \vspace{3pt}\\
L_{2} (U,V) := -d \Delta V - 2 d \alpha \partial_x V + (-d \alpha^2 - f'(0)) V , \vspace{3pt}\\
E(U,V):= -d \partial_y V(x,0) + \nu (x) V(x,0) - \mu (x) U(x).
\end{array}
\end{equation*}
Then, plugging the previous ansatz into \eqref{eq:model_lin} leads to the following eigenvalue problem
\begin{equation}\label{eq:principal_eigenvalue}
\left\{
\begin{array}{ll}
L_1 (U_\alpha,V_\alpha) = \Lambda U_\alpha  & \mbox{ in } \mathbb{R} ,\vspace{3pt}\\
L_2 (U_\alpha,V_\alpha) = \Lambda V_\alpha  & \mbox{ in } \mathbb{R} \times (0,+\infty),\vspace{3pt}\\
E(U_\alpha,V_\alpha)=0 & \mbox{ in } \mathbb{R},
\end{array}
\right.
\end{equation}
together with the conditions
$$\left\{
\begin{array}{ll}
U_\alpha \mbox{ is $L$-periodic and positive},\vspace{3pt}\\
V_\alpha \mbox{ is $L$-periodic and positive},\\\end{array}
\right.
$$
and
$$\Lambda  = -\alpha c.$$
In this section, we will deal with the well-posedness of this eigenvalue problem. The main difficulty comes from the unboundedness of the domain, which typically entails the non-uniqueness of the eigenvalue associated with positive eigenfunctions; thus the need of some well-chosen "generalized"  eigenvalues. We refer to~\cite{BR-general_eigen} and the references therein for many results on such eigenvalues and their applications (for single equations). 

Here, we will construct a generalized principal eigenvalue by first truncating the domain in the $y$-direction, then passing to the limit. This will lead us to the following theorem.

\begin{theo}\label{th:eigenvalue}
There exists a concave (hence continuous) function $\alpha \geq 0 \mapsto \Lambda (\alpha)$ such that :
\begin{enumerate}[$(i)$]
\item the system \eqref{eq:principal_eigenvalue} with $\Lambda = \Lambda (\alpha)$ admits a positive and $L$-periodic solution $(U_\alpha,V_\alpha)$ such that, up to some normalization,
$$U_\alpha (x) \leq 1 , \quad  V_\alpha (x,y) \leq C (1+y),$$
for some $C>0$ (possibly depending on $\alpha$);
\item the system \eqref{eq:principal_eigenvalue} admits no positive and $L$-periodic solution for any $\Lambda > \Lambda (\alpha)$.
\end{enumerate}
\end{theo}
The rest of this section is devoted to the proof of Theorem~\ref{th:eigenvalue}.
\subsection{Principal eigenvalue on truncated fields}

Here and henceforth, $\alpha \geq 0$ is a fixed parameter. As announced, we first seek for a principal eigenvalue in the truncated fields
$$\Omega_R := \mathbb{T} \times (0,R),$$
where $\mathbb{T}$ denotes with some slight abuse of notations the one dimensional torus $\R_{/L \mathbb{Z}}$.

More precisely, we seek $\Lambda_R (\alpha) \in \R$ such that there exists 
$$(U_{\alpha,R},V_{\alpha,R})>0 \mbox{ periodic with respect to $x$}$$
solution of the following system:
\begin{equation}\label{eq:truncated_eigen}\left\{
\begin{array}{ll}
L_1 (U_{\alpha,R},V_{\alpha,R}) =\Lambda_R (\alpha) U_{\alpha,R}  & \mbox{ in } \mathbb{T} ,\vspace{3pt}\\
L_2 (U_{\alpha,R},V_{\alpha,R}) =\Lambda_R (\alpha) V_{\alpha,R}  & \mbox{ in } \Omega_R ,\vspace{3pt}\\
E(U_{\alpha,R},V_{\alpha,R})=0 = V_{\alpha,R}|_{y=R}  & \mbox{ in } \mathbb{T}.
\end{array}
\right.
\end{equation}
The compactness of the domain, together with the cooperative nature of the road-field system, now provides a much more convenient mathematical framework where the classical Krein-Rutman theory applies. However, because of the non-standard coupling through a boundary condition, and anticipating the need to pass to the limit as $R \to +\infty$, we include the details of the proof.
\begin{prop}\label{prop:truncated1}
There exists $M_\alpha$ such that, for any $M >M_\alpha$ and given any functions $ g_1 \in C^{0,r}(\mathbb{T})$ and $ g_2 \in C^{0,r}(\overline{\Omega_R})$, there exists a unique solution $U,V$ to
\begin{equation}\label{sys_data}\left\{
\begin{array}{ll}
L_1  (U,V)+MU = g_1  & \mbox{ in } \mathbb{T} ,\vspace{3pt}\\
L_2 (U,V) + MV =g_2  & \mbox{ in } \Omega_R ,\vspace{3pt}\\
E(U,V)=0 = V|_{y=R} & \mbox{ in } \mathbb{T},
\end{array}
\right.
\end{equation}
 with $U\in C^{2,r}(\mathbb{T})$ and $V\in C^{2,r}(\Omega_R)\cap C^{1,r}(\overline{\Omega_R})$. 

Moreover, \eqref{sys_data} enjoys a strong maximum principle: $g_1,g_2\geq 0$ implies $U,V\geq 0$, and if either $g_1\not\equiv 0$ or $g_2\not\equiv 0$ then $U>0$ in $\mathbb{T}$ and $V>0$ in $\mathbb{T}\times[0,R)$.
\label{prop:L_Muv=fg_solvable}
 \end{prop}

\begin{proof}
Our proof is divided into three steps. The first step consists in constructing families of sub and supersolutions, which will serve in proving both the maximum principle (step 2), and the existence of a solution (step 3).

 {\it Step 1: construction of barriers.} We first seek positive supersolutions of the form
 $$
 \left(\overline{U}(x),\overline{V}(x,y)\right) =K \left(1,C(1+e^{-\omega y})\right)
 $$
 for some large $K>0$. Choosing for example $\omega=1$, $C=\mu_1/d$, an explicit computation shows that if
 \begin{equation}
 M>M_{\alpha}:=\max\left\{d(\alpha^2+1)+f'(0),D\alpha^2+2\frac{\mu_1\nu_1}{d}\right\}
 \label{eq:def_Malpha}
 \end{equation}
 then we can choose $K>0$ large enough (depending only on $\max g_1$ and $\max g_2$ if positive) such that the pair $(\overline{U},\overline{V})>0$ is indeed a supersolution of~\eqref{sys_data}. By linearity and up to increasing $K$ (depending on $\min g_1$ and $\min g_2$ if negative), $\left(\underline{U},\underline{V}\right)=-\left(\overline{U},\overline{V}\right)<0$ is also a negative subsolution. Moreover, it is easy to check that these are \emph{strict} sub and supersolutions, i.e. they are not solutions.
 
 {\it Step 2: strong maximum principle and uniqueness.} For fixed data $g_1,g_2\geq 0$, let $(U,V)$ be any solution of \eqref{sys_data}. One can then easily check that $(\underline{U},\underline{V})$, as defined above, is a strict subsolution of \eqref{sys_data} for any $K>0$. Now assume by contradiction that $U$ or $V$ take negative values somewhere in their respective domain. Because the previous subsolutions $(\underline{U},\underline{V})$ are negative up to the boundary we can define
$$
\theta^* = \min \{\theta > 0 \ : \quad  (U,V) \geq  \theta (\underline{U},\underline{V}) \} \in (0,+\infty).
$$
Then $(U-\theta^*\underline{U},V-\theta^*\underline{V})\geq 0$ is a supersolution of \eqref{sys_data} and attains zero either on the road or in the field. In the former case, the strong maximum principle implies that $U \equiv \theta^* \underline{U}$, which in turn implies that $V(y=0) \equiv \theta^* \underline{V} (y=0)$ and, thanks to the Hopf lemma, $V \equiv \theta^* \underline{V}$. This contradicts the fact that $(\underline{U},\underline{V})$ is not solution of \eqref{sys_data}. A similar argument leads to the same contradiction in the latter case, and we conclude that $(U,V) \geq 0$.

The same argument as above shows that $(U,V)>0$ as soon as $g_1 \not\equiv 0$ or $g_2 \not\equiv 0$ and we omit the details. Moreover, by the linearity of \eqref{sys_data}, it immediately follows that $(0,0)$ is the unique solution when $g_1 \equiv 0$ and $g_2 \equiv 0$ and, therefore, the system~\eqref{sys_data} always admits at most one solution.

{\it Step 3: existence.} Arguing as in the proof of Theorem~\ref{thm:cauchy} it is easy to see that the linear evolution problem
$$
\left\{
\begin{array}{ll}
\partial_t u = -(L_1 +M) (u,v) + g_1  & \mbox{ in } \mathbb{T} \times (0,+\infty) ,\vspace{3pt}\\
\partial_t v = - (L_2 +M) (u,v) +g_2  & \mbox{ in } \Omega_R \times (0,+\infty),\vspace{3pt}\\
E(u,v)=0 = v|_{y=R} & \mbox{ in } \mathbb{T} \times (0,+\infty), \vspace{3pt}\\
(u,v)_{t=0}=(\underline{U},\underline{V}),
\end{array}
\right.
$$
is solvable. Since the initial data $(\underline{U},\underline{V})=-(\overline{U},\overline{V})$ is a subsolution of the stationary problem the solution $(u ,v)$ is non-decreasing in time, and by the maximum principle
$$
\forall\,t\geq 0:\qquad (\underline{U},\underline{V})\leq (u,v) \leq (\overline{U},\overline{V}).
$$
Letting $t\to\infty$ we get by standard parabolic estimates and monotonicity that $(u,v)$ converges to a stationary solution $(U,V)$ of \eqref{sys_data}. Standard elliptic estimates give the desired regularity for $U(x),V(x,y)$ and the proof is complete.
%
%
\end{proof}

According to Proposition~\ref{prop:truncated1} above, the system~\eqref{sys_data} is well-posed for $M$ large enough. By the Krein-Rutman theory, we will get as a consequence that:
\begin{prop}\label{prop_KR}
For any $R >0$ and $\alpha \geq 0$ there exists a unique principal eigenvalue $\Lambda_R (\alpha)$ of \eqref{eq:truncated_eigen} associated with positive and periodic eigenfunctions $(U_{\alpha,R},V_{\alpha,R})$ of \eqref{eq:truncated_eigen}, which are also unique up to multiplication by a positive factor.

Furthermore, $- \Lambda_R (\alpha) \leq M_{\alpha}$,
where $M_\alpha$ is defined by Proposition~\ref{prop:truncated1} and more precisely in~\eqref{eq:def_Malpha}.
\label{prop:exists_eigenvalue+upper_estimate}
\end{prop}
\begin{proof}
Let us fix any $0 \leq r' < r$. Define first the Banach space
$$X:=\{(U,V)\in C^{1,r'}(\mathbb{T})\times  C^{1,r'}(\mathbb{T}\times[0,R]):\quad  V|_{y=R}=0 \}
$$
with its natural H\"{o}lder norm, and the positive cone
$$
X^+:=\{(U,V)\in C^{1,r'}(\mathbb{T})\times C^{1,r'}(\mathbb{T}\times[0,R]):\quad U,V \geq 0,\  V|_{y=R}=0\}.
$$
Note that $X^+$ has non-empty interior in $X$ (take for instance $U(x)=1$ and $V (x,y) =1-y/R$).

Choose now any $M>M_{\alpha}$: by Proposition~\ref{prop:L_Muv=fg_solvable} one can define the operator $T: X \mapsto X$, which to any given pair $(g_1,g_2)$ associates the unique solution $(U,V)$ of \eqref{sys_data}. Note that $T$ is compact. This follows from elliptic estimates~\cite{GilbargTrudinger} and the fact that, for any functions $g_1$ and $g_2$, the $L^\infty$ norm of $(U,V)=T(g_1,g_2)$ is bounded from above by some constant which only depends on $\|g_1\|_\infty$ and $\|g_2\|_\infty$ (see the proof of Proposition~\ref{prop:truncated1} above). Moreover, using again Proposition~\ref{prop:L_Muv=fg_solvable}, $T (X^+) \subset X^+$, and $T$ is even strongly positive in the sense that : if $(g_1,g_2) \in X^+\setminus \{(0,0)\}$ and $(U,V)=T(g_1,g_2)$, then $(U,V) >0$ for $y\in [0,R)$ as well as $\partial_y V (x,R)<0$ (by Hopf lemma), which implies that $(U,V)$ is an interior point of $X^+$.

By the Krein-Rutman theorem we conclude that there exists a positive eigenvalue $\sigma_M>0$ of $T$ associated with positive (and bounded) eigenfunctions $U_M,V_M>0$. Transferring as usual in terms of the original PDE we have therefore an eigenvalue $\Lambda_R (\alpha) = \left( \frac{1}{\sigma_M} - M \right)$ of \eqref{eq:truncated_eigen}, associated with the same positive and periodic eigenfunctions.

It now remains to check that such a $\Lambda_R (\alpha)$ is unique (in particular, it did not depend on the choice of $M > M_\alpha$ above). Let $\Lambda_1$ and $\Lambda_2$ be two eigenvalues of \eqref{eq:truncated_eigen} associated with positive and periodic eigenfunctions, respectively $(U_1, V_1)$ and $(U_2, V_2)$, and set
$$
\theta^* := \min \{ \theta >0 \ : \quad \theta (U_{1},V_{1}) \geq (U_{2}, V_{2}) \}.
$$
If $\theta^* V_1 \equiv V_2 $, then trivially $\Lambda_{1} = \Lambda_{2}$ (the field equation does not depend on $U$). Otherwise, we know that $\theta^* (U_{1},V_{1}) \geq (U_{2} , V_{2})$ and, by Hopf lemma and $x$-periodicity, 
$$
\partial_y (\theta^* V_{1} -V_{2})|_{y=R} <0.
$$
Hence, by construction, either $\theta^* U_{1} = U_{2}$ for some $x_0 \in \mathbb{T}$, or $\theta^* V_{1} = V_{2}$ for some $(x_0, y_0) \in \mathbb{T} \times [0,R)$. 

In the former case, then
$$0 \geq \theta^* L_1 (U_{1},V_{1}) (x_0) - L_1 (U_{2} , V_{2}) (x_0) = (\Lambda_{1} - \Lambda_{2}) U_{2} (x_0),$$
hence $\Lambda_{1} \leq \Lambda_{2}$. One gets the same conclusion in the latter case if the contact point occurs in the interior, namely $y_0 >0$. Otherwise, $\theta^* V_1 > \theta^* V_2$ in $\Omega_R$ and $\theta^* V_1 (x_0,0) = V_2 (x_0,0)$, which leads to a contradiction by Hopf lemma. We conclude that $\Lambda_1 \leq \Lambda_2$, and thus by symmetry $\Lambda_1 = \Lambda_2$. Using the usual combination of the strong maximum principle and Hopf lemma, it is now straightforward to check that $\theta^* (U_1, V_1) \equiv (U_2,V_2)$, namely that the principal eigenfunction pair is unique up to multiplication by a positive factor.

Finally, we end the proof by noting that the upper bound $-\Lambda_R (\alpha) \leq M_{\alpha}$ follows from the fact that $\Lambda_R (\alpha) = \frac{1}{\sigma_M} - M > -M$ for any $M > M_\alpha$.
\end{proof}

\subsection{Passage to the whole field}

Before we pass to the limit as $R \to +\infty$, we need some monotonicity and bounds on the function $R \mapsto \Lambda_R (\alpha)$. This is the purpose of the proposition below:
\begin{prop}
For any $\alpha\geq 0$, the function $R>0 \mapsto -\Lambda_R (\alpha)$ is increasing. Moreover,
$$\max \left\{ d\alpha^2 + f'(0) - d\frac{\pi^2}{R^2} \; , \; D\alpha^2 - \lambda_\alpha \right\} <- \Lambda_R (\alpha) \leq M_\alpha, $$
where $M_\alpha$ was defined in Proposition~\ref{prop:truncated1}, and $\lambda_\alpha\in[\min\mu,\max\mu]$ is the principal eigenvalue of $-D\frac{d^2}{dx^2}-2\alpha D\frac{d}{dx}+\mu(x)$ on the torus.
\label{prop:monotonicity_wrt_R}
\end{prop}
\begin{proof}
Let us briefly sketch the proof of the monotonicity with respect to $R$. Fix $R_1 > R_2 >0$ and let $\Lambda_1 = \Lambda_{R_1} (\alpha)$, $\Lambda_2 = \Lambda_{R_2} (\alpha)$ be associated with the positive and periodic eigenfunctions $(U_1,V_1)$ and $(U_2,V_2)$. Since $R_1 > R_2$, one can proceed as before and, multiplying $(U_1,V_1)$ by a well-chosen constant $\theta>0$, assume without loss of generality that $(U_1,V_1)\geq(U_2,V_2)$ in $\Omega_{R_2} =\mathbb{T}\times(0,R_2)$ and that there is a contact point either between $U_1,U_2$ at some $x_0 \in\mathbb{T}$, or between $V_1,V_2$ at some $(x_0,y_0)\in \overline{\Omega_{R_2}}$ (or both simultaneously). Proceeding exactly as in the proof of Proposition~\ref{prop:exists_eigenvalue+upper_estimate}, we conclude that $\Lambda_1 \leq \Lambda_2$. Furthermore, since $V_1 \not \equiv V_2$ (thanks to the Dirichlet condition of $V_2$ at $y=R_2$), it even follows from the strong maximum principle that $\Lambda_1 < \Lambda_2$.

Let us now prove the lower bounds on $-\Lambda_R (\alpha)$ (the upper bound was proved in Proposition~\ref{prop:exists_eigenvalue+upper_estimate}). Let $\tilde{U}_{\alpha}(x)>0$ be the principal eigenfunction associated with $\lambda_\alpha$. Letting $z=\frac{U_{\alpha,R}}{\tilde{U}_\alpha}$ one can check that $L_{1}(U_{\alpha,R},V_{\alpha,R})=\Lambda_R(\alpha)U_{\alpha,R}$ in \eqref{eq:truncated_eigen} can be rewritten as
$$-D z'' - 2D \left( \alpha + \frac{\tilde{U}_\alpha ' }{\tilde{U}_\alpha} \right) z'  + (-D\alpha^2 + \lambda_\alpha - \Lambda_R (\alpha) ) z = \nu \frac{V_{\alpha,R}}{\tilde{U}_\alpha} \geq 0.
$$
We proceed by contradiction and assume that $-D\alpha^2 + \lambda_\alpha - \Lambda_R (\alpha) \leq 0$. Then any positive constant is a subsolution of the above equation satisfied by $z$. In particular, since $z>0$ is periodic, we infer that it is identically equal to its minimum and thus constant. It then follows that
$$\nu \frac{V_{\alpha,R}}{\tilde{U}_\alpha} = (-D\alpha^2 + \lambda_\alpha - \Lambda_R (\alpha)) z\leq 0,$$
which contradicts the positivity of $V_{\alpha,R}$.

It only remains to prove the other lower bound, namely
$$-\Lambda_R (\alpha) > d\alpha^2 + f'(0) - d \frac{\pi^2}{R^2}.$$
For fixed $R>0$, denote $\omega_R=\frac{\pi}{R}$. Arguing again by contradiction, if our lower bound does not hold then
 $$
 \omega:=\sqrt{\frac{d\alpha^2+f'(0)+\Lambda_R(\alpha)}{d}}\geq\omega_R>0.
 $$
 Averaging in $x$ the equation for $V_{\alpha,R}$ we see that $\Phi_R (y):=\int_{\mathbb{T}}V_{\alpha,R}(x,y)dx$ solves
 $$
 \Phi_R ''(y)+\omega^2\Phi_R(y)=0.
$$
Since $\Phi_R (y) >0$ for $y \in [0,R)$ and $\Phi_R (R)=0$ (recall that the eigenfunction $V_{\alpha,R} (x,y)$ is positive up to the road and satisfies zero Dirichlet boundary conditions at $y=R$), we get that 
$$\Phi_R (y) = C \sin (\omega (R-y) )$$
for some constant $C >0$. 
If $\omega\geq \omega_R$ then $[0,R)$ contains at least half a period of $\Phi_R$, hence $\Phi_R$ must have a non-positive value in $[0,R)$. This contradicts the strict positivity up to the road and the proposition is proved.
\end{proof}

We can now infer from the previous proposition that for any $\alpha$ fixed $\Lambda_R (\alpha)$ converges as $R \to +\infty$, and we thus define the generalized principal eigenvalue
\begin{equation}\label{eigen_def}
\Lambda (\alpha) := \lim_{R \to +\infty} \Lambda_R (\alpha).
\end{equation}
It remains to prove that the associated pair of eigenfunctions also converges to a non-trivial limit as $R \to +\infty$. From now on, we will denote by $(U_{R},V_{R})$ and $\Lambda_{R}$ the principal eigenfunction pair and eigenvalue of the truncated problem (omitting the $\alpha$ dependence for convenience).
\begin{lem}\label{lem:lim1}
Assume that $\|U_{R}\|_{L^{\infty}(\mathbb{T})}=1$, then there exists a positive constant $C_1$, which is independent of $R$ large enough,  such that $$\|V_R (\cdot,0)\|_{L^\infty (\mathbb{T})} > C_1.$$
\end{lem}
\begin{proof}
We argue by contradiction. If the proposition does not hold, then we can find a sequence $R_{k}\to\infty$ such that 
$$
\|V_{R_{k}}(\cdot,0)\|_{L^\infty(\mathbb{T})}\rightarrow 0.
$$
Moreover, we assumed that $\|U_R\|_{L^\infty (\mathbb{T})} =1$. In particular, for all $k$, there exists $x_k$ such that
$$\|U_{R_k} \|_{L^\infty (\mathbb{T})} = U (x_k) = 1.$$
Then the sequences of functions $(U_{R_k})_k$ and $(V_{R_k} (\cdot,0))_k$ are uniformly bounded in $L^\infty (\mathbb{T})$. By standard elliptic estimates~\cite{ GilbargTrudinger}, one can extract a subsequence such that $U_{R_{k}}\rightarrow U_{\infty}$ and $x_k \to x_\infty$ as $k \to +\infty$, where $U_\infty \geq 0$ satisfies 
$$-DU_{\infty}''-2D\alpha U_{\infty}'+(-D\alpha ^{2}+\mu(x))U_{\infty}=\Lambda (\alpha) U_{\infty},$$
together with
$$U_\infty (x_\infty) =1.$$
By the strong maximum principle $U_\infty$ is positive, and by construction it is also periodic. By uniqueness of the principal eigenvalue $\lambda_\alpha$ of the periodic operator
$$-D U'' - 2D\alpha U' + \mu (x) U,$$
we get that
$$\Lambda (\alpha )=-D\alpha^{2}+\lambda_{\alpha}.$$
Because $\Lambda_{R} (\alpha)$ is decreasing with respect to $R$, it follows that for all $R >0$:
$$\Lambda_{R} (\alpha )>-D\alpha^{2}+\lambda_\alpha,$$
which contradicts the previous proposition.
\end{proof}

\begin{lem}\label{loc bounded}
Normalizing with $\|V_{R}(\cdot ,0)\|_{L^{\infty}(\mathbb{T})}=1$, we have the upper estimate
$$
0\leq V_{R}(x,y)\leq 1+\frac{\nu_1}{d}y,\qquad (x,y)\in\overline{\Omega_R}.
$$
\end{lem}
\begin{proof}
Note that, if there exists some $R_{0}$ such that $\Lambda_{R_{0}}(\alpha) +d\alpha^{2}+f'(0)\leq 0$, then by monotonicity $\Lambda_{R}(\alpha) +d\alpha^{2}+f'(0)< 0$ for all $R>R_0$. In this case, one can apply the maximum principle and conclude that $V_R$ cannot reach a positive interior maximum. Therefore, recalling that $0\leq V_R|_{y=0}\leq 1$ and $V_R|_{y=R}=0\leq 1$, we get that $V_R (x,y) \leq 1$ and our statement holds.

However, we only know that $\Lambda_R (\alpha) + d\alpha^2 + f'(0) - d\frac{\pi^2}{R^2} < 0$ by Proposition~\ref{prop:monotonicity_wrt_R}. Thus it remains to consider the case
$$\omega^2 = 
\frac{\Lambda_R (\alpha) + d\alpha^2 + f'(0)}{d} > 0
$$
with $0<\omega<\omega_R:=\frac{\pi}{R}$. We denote again the average $\Phi_{R}(y)=\int _{\mathbb{T}}V_{R}(x,y) dx$, which satisfies as before $\Phi_R''+\omega^2\Phi_R =0$ and thus
$$
\Phi_R (y) = C \sin (\omega (R-y) )
$$
for some $C>0$. The fact that $0<\omega<\omega_R=\pi/R$ of course ensures that $\Phi_R (y) $ is positive for $y \in[0,R)$,
but we also require that
$$d\Phi_{R}'(0)< \nu_1 \Phi_{R}(0),$$
because of $d\partial_y V(x,0)=\nu(x)V(x,0)-\mu(x)U(x)\leq \nu_1 V(x,0)$. Thus,
$$- \omega \cotan  (\omega R) \leq \frac{\nu_1}{d} .$$
It is then easy to see that $\omega \in [0,\overline{\omega}]$, where $\overline{\omega}=\overline{\omega}(R)$ is the unique solution of
$$
-\overline{\omega} \cotan(\overline{\omega}R ) =\frac{\nu_1}{d}>0
$$
in $(0, \frac{\pi}{R})$.

Introducing the function
 $$
 \overline{V}_{R}(x,y):=\frac{1}{\sin \left(\overline{\omega}R\right)}\sin \left(\overline{\omega}(R-y) \right),
 $$
we claim now that this gives an upper barrier, that is
\begin{equation}\label{claim42}
 \forall (x,y)\in\mathbb{T}\times[0,R]:\qquad V_R(x,y)\leq\overline{V}_R(y).
\end{equation}
Note first that $V_R$, $\overline{V}_R>0$ for $y\in[0,R)$ and both are $C^1$ up to the boundary $\{y=R\}$ where $\partial_y\overline{V}_R<0$, thus
$$
\theta^*:=\min\{ \theta>0\ : \quad \ \theta \overline{V}_{R}\geq V_{R} \mbox{ for }(x,y)\in \Omega_{R}\} \geq 1
$$
is well-defined. The inequality $\theta^* \geq 1$ immediately follows from our normalization on the road.

We establish now \eqref{claim42}. Assuming by contradiction that $\theta^*>1$ and letting the elliptic operator $\overline{L}=-\Delta-2\alpha\partial_x$, then $\overline{L}[V_R]=\omega^2V_R$ and $\overline{L}[\overline{V}_R]=\overline{\omega}^2\overline{V}_R$. As a consequence $z:=\theta^*\overline{V}_R-V_R\geq 0$ satisfies
$$
\overline{L}[z]=\overline{\omega}^2\theta^*\overline{V}_R-\omega^2V_R=(\overline{\omega}^2-\omega^2)\theta^*\overline{V}_R+\omega^2(\theta^*\overline{V}_R-V_R)\geq 0\qquad\text{in }\Omega_R ,
$$
because $\omega\leq\overline{\omega}$ and $z=\theta^*\overline{V}_R-V_R\geq 0$. 

Since $\overline{V}_R\geq V_R$ on the road and as $\theta^*>1$ clearly $z|_{y=0} >0$, so the strong maximum principle shows that $z>0$ for $y\in [0,R)$. By the Hopf lemma and $x$-periodicity we also get $\partial_y z|_{y=R}<0$. This easily implies that $z\geq \eps \overline{V}_R$ in $\Omega_R$ for some small $\eps>0$ and $(\theta^*-\eps)\overline{V}_R\geq V_R$, which in turn contradicts the minimality of $\theta^*$ and thus entails our claim \eqref{claim42}.

Recalling that $0<\overline{\omega}<\omega_R=\pi/R$, we see that $\overline{V}_R(y)=C\sin(\overline{\omega}(R-y))$ is concave in $y\in[0,R]$. Thus by definition of $\overline{\omega}$:
$$
V_R (x,y) \leq \overline{V}_R(y)\leq \overline{V}_R(0)+\overline{V}_R'(0)y=1-\overline{\omega} \cotan(\overline{\omega}R) y=1+\frac{\nu_1}{d}y
$$
for all $y \in [0,R]$ and $x \in \mathbb{T}$, and the proof is complete.
\end{proof}
\begin{lem}\label{lem:lim3}
Normalizing with $\|U_{R}\|_{L^{\infty}(\mathbb{T})}=1$, we have
$$
\|V_R (\cdot,0)\|_{L^\infty (\mathbb{T})} \leq  C_2
$$
for some $C_2>0$ independent of $R$.
\end{lem}
\begin{proof}
Assume by contradiction that our statement does not hold. Suitably normalizing, we can therefore assume that $\|V_{R_j}(\cdot,0)\|_{L^{\infty}(\mathbb{T})}=1$ and $\|U_{R_{j}}\|_{L^{\infty}(\mathbb{T})}\to 0$ for some subsequence $R_j \to +\infty$. By \eqref{eigen_def}, Lemma \ref{loc bounded}, and using standard elliptic estimates as before, we get up to extraction of a subsequence that $V_{R_j}$ converges (locally uniformly on compact sets of $\overline{\Omega}$) to a non-trivial function $V_{\infty}\geq 0$, which satisfies
\begin{equation}
\begin{cases}
-d\Delta V_\infty -2d\alpha\partial_{x}V_\infty -(d\alpha^{2}+f'(0)+\Lambda(\alpha))V_\infty=0 ,&x\in\mathbb{T},y>0, \vspace{3pt}\\
d\partial _{y}V_\infty (x,0)=\nu(x)V_\infty (x,0),&x\in\mathbb{T}.
\end{cases}
\label{eq:lim_V_U=0}
\end{equation}
Moreover, standard elliptic estimates also allow to pass to the limit on the road equation and get
$$
\nu (x) V_\infty (x,0) \equiv 0.
$$
Choose now any $x_0$ such that $\nu (x_0) >0$. Then $V_{\infty}(x_0,0)=0$, and by Hopf lemma $\partial_yV_{\infty}(x_0,0)>0$. This contradicts the boundary condition in \eqref{eq:lim_V_U=0} and the proof is complete.
\end{proof}

\begin{proof}[Proof of Theorem~\ref{th:eigenvalue}]
We recall that $\Lambda (\alpha)$ was defined as the limit of $\Lambda_R (\alpha)$ as $R \to +\infty$ in~\eqref{eigen_def}, and postpone the proof of its concavity to the next subsection.

Combining the above Lemmas~\ref{lem:lim1}, \ref{loc bounded} and \ref{lem:lim3}, together with elliptic estimates~\cite{GilbargTrudinger}, we can now pass to the limit as $R \to +\infty$. We get a pair of non-negative and periodic eigenfunctions $(U_\alpha,V_\alpha)$ of \eqref{eq:principal_eigenvalue} in the whole field with $\Lambda = \Lambda (\alpha)$, which satisfy
$$
U_\alpha (x) \leq 1 , \quad V_\alpha (x,y) \leq C (1+y),
$$
after a suitable renormalization. The fact that these eigenfunctions are positive is a straightforward application of the strong maximum principle and Hopf lemma, as we already used extensively, and part $(i)$ in Theorem~\ref{th:eigenvalue} is proved.

Let us now briefly check part $(ii)$. Let some $\Lambda$ be such that there exists a positive and periodic eigenfunction $(U,V)$ of \eqref{eq:principal_eigenvalue}. Proceeding as in the proof of Proposition~\ref{prop:monotonicity_wrt_R} (more precisely, the proof of the monotonicity of $\Lambda_R$ with respect to~$R$), one can check that $\Lambda < \Lambda_R (\alpha)$ for any $R >0$. Passing to the limit $R\to +\infty$, it immediately follows that $\Lambda \leq \Lambda (\alpha)$.
\end{proof}
%
\subsection{Further properties}
In order to complete the proof of Theorem~\ref{th:eigenvalue}, it only remains to prove the concavity of $\Lambda (\alpha)$, which will play an important role in the study of the spreading speeds for exponentially decaying initial data. We exploit again the construction of $\Lambda (\alpha)$ via the more convenient framework of truncated problems. 
\begin{prop}\label{concavity}
The functions $\alpha \mapsto \Lambda (\alpha)$ and $\alpha \mapsto \Lambda_R (\alpha)$, for all $R>0$, are concave.
\end{prop}
\begin{proof}
Since the pointwise limit of concave functions is again a concave function, we only need to prove the concavity of $\Lambda_R (\alpha)$. Following the steps of \cite[Proposition~5.7]{BH-front}, we begin by showing that
\begin{equation}\label{eqn:concav_formula}
\Lambda_R (\alpha) = \sup_{(\phi, \psi) \in \mathcal{E}} \min \left\{ \inf_{\mathbb{T}} \frac{L_1 (\phi,\psi)}{\phi} , \inf_{\Omega_R} \frac{L_2 (\phi,\psi)}{\psi} \right\},
\end{equation}
where
$$\begin{array}{l}
\mathcal{E}= \left\{ (\phi,\psi) \in C^2 (\mathbb{T}) \times( C^2 (\Omega_R) \cap C^1 (\overline{\Omega_R}) )\ : \right.  \vspace{5pt} \\
\hspace{2cm} \left. \phi>0, \psi>0 \mbox{ in } \mathbb{T} \times [0,R), \ \psi (y=R) =0 > \partial_y  \psi (y=R) , \ E (\phi,\psi) \geq 0  \right\}.
\end{array}
$$
Note that the pair $(U_{\alpha,R} , V_{\alpha,R})\in\mathcal{E}$ and, therefore, 
$$
\Lambda_R (\alpha) \leq \sup_{(\phi, \psi) \in \mathcal{E}} \min \left\{ \inf_{\mathbb{T}} \frac{L_1 (\phi,\psi)}{\phi} , \inf_{\Omega_R} \frac{L_2 (\phi,\psi)}{\psi} \right\}.
$$
Now proceed by contradiction and assume that there exists $(\phi,\psi) \in \mathcal{E}$ such that
$$
\Lambda_R (\alpha) < \min \left\{ \inf_{\mathbb{T}} \frac{L_1 (\phi,\psi)}{\phi} , \inf_{\Omega_R} \frac{L_2 (\phi,\psi)}{\psi} \right\}.
$$
Proceeding as before and thanks to the definition of the admissible set $\mathcal{E}$, one can find some critical $\theta >0$ such that $U_{\alpha,R} \leq \theta \phi$ and $V_{\alpha,R} \leq \theta \psi$ with either $U_{\alpha,R} (x_0) = \theta \phi (x_0)$, $V_{\alpha,R} (x_0,y_0) = \theta \psi (x_0,y_0)$ or $\partial_y V_{\alpha,R} (x_0, R) = \theta \partial_y \psi (x_0,R)$ for some $x_0 \in \mathbb{T}$ and $y_0 \in [0,R)$.

Note that $(\theta \phi - U_{\alpha,R}, \theta \psi - V_{\alpha,R}) \geq 0$ satisfies
\begin{equation*}
\left\{
\begin{array}{l}
L_1 (\theta \phi - U_{\alpha,R}, \theta \psi - V_{\alpha,R})  > \Lambda_R (\alpha) (\theta \phi - U_{\alpha,R}) , \vspace{3pt}\\
L_2 (\theta \phi - U_{\alpha,R}, \theta \psi - V_{\alpha,R}) > \Lambda_R (\alpha) (\theta \psi - V_{\alpha,R}), \vspace{3pt}\\
E(\theta \phi - U_{\alpha,R}, \theta \psi - V_{\alpha,R}) \geq 0, \vspace{3pt}\\
(\theta \psi - V_{\alpha,R}) (y=R) = 0.
\end{array}
\right.
\end{equation*}
By the strong maximum principle and the Hopf lemma as before, it follows that $\theta \phi \equiv U_{\alpha,R}$ and $\theta \psi \equiv V_{\alpha,R}$, which is a contradiction. Hence, \eqref{eqn:concav_formula} is proved.

Let us now proceed to the proof of concavity. We first introduce 
$$
\mathcal{E}_\alpha =  \{ (\tilde{\phi},\tilde{\psi})  \ : \quad \exists (\phi,\psi) \in \mathcal{E} \mbox{ such that } \tilde{\phi} = e^{\alpha x} \phi  \mbox{ and } \tilde{\psi} = e^{\alpha x} \psi \}.
$$
It is then clear that
\begin{eqnarray*}
\Lambda_R (\alpha) &= & \sup_{(\phi, \psi) \in \mathcal{E}} \min \left\{ \inf_{\mathbb{T}} \frac{L_1 (\phi,\psi)}{\phi} , \inf_{\Omega_R} \frac{L_2 (\phi,\psi)}{\psi} \right\}\\
&=& \sup_{(\tilde{\phi}, \tilde{\psi}) \in \mathcal{E}_\alpha} \min \left\{ \inf_{\mathbb{R}} \frac{-D  \tilde{\phi}'' + \mu \tilde{\phi} - \nu \tilde{\psi}}{\tilde{\phi}} , \inf_{\R \times [0,R)} \frac{-d \tilde{\psi}'' -f'(0) \tilde{\psi}}{\tilde{\psi}} \right\}.
\end{eqnarray*}
Let $\alpha_1 \geq 0$ and $\alpha_2 \geq 0$, and choose any $(\tilde{\phi}_1,\tilde{\psi}_1)=e^{\alpha_1 x}(\phi_1,\psi_1) \in \mathcal{E}_{\alpha_1}$, $(\tilde{\phi}_2 , \tilde{\psi}_2)=e^{\alpha_2x}(\phi_2,\psi_2) \in \mathcal{E}_{\alpha_2}$ with $(\phi_i,\psi_i)\in\mathcal E$. Fixing any $t\in (0,1)$ and defining $\alpha = t\alpha_1 + (1-t) \alpha_2$, we claim that
$$
(\tilde{\phi},\tilde{\psi}) := (e^{t \ln \tilde{\phi}_1+ (1-t) \ln \tilde{\phi}_2 } , e^{t \ln \tilde{\psi}_1 + (1-t) \ln \tilde{\psi_2}})=e^{\alpha x}\left(\phi_1^t\phi_2^{1-t},\psi_1^t\psi_2^{1-t}\right)\in \mathcal E_\alpha.
$$
Indeed since $(\phi_i,\psi_i)\in \mathcal{E}$ we have that $\psi_1,\psi_2$ both vanish at $y=R$ with non-zero slopes $p_1 (x) := \partial_y \psi_1 (x,y=R)$, $p_2 (x):= \partial_y \psi_2 (x,y=R)$, and it is then easy to check that $\psi=\psi_1^{t}\psi_2^{1-t}$ also vanishes at $y=R$ with non-zero slope $p=p_1^tp_2^{1-t}$ and $\psi\in \mathcal{C}^1(\overline{\Omega_R})$. The other conditions for $(\phi,\psi)=(\phi_1^t\phi_2^{1-t},\psi_1^t\psi_2^{1-t})\in\mathcal{E}$ also follow from straightforward computations, thus $(\tilde\phi,\tilde\psi)=e^{\alpha x}(\phi,\psi)\in\mathcal E_{\alpha}$ as claimed.
 
Then one can check that
\begin{align*}
&\min \left\{ \inf_{\mathbb{R}} \frac{-D  \tilde{\phi}'' + \mu \tilde{\phi} - \nu \tilde{\psi}}{\tilde{\phi}} , \inf_{\R \times [0,R)} \frac{-d \tilde{\psi}'' -f'(0) \tilde{\psi}}{\tilde{\psi}} \right\}\\
&\hspace{2cm} \geq  t \min \left\{ \inf_{\mathbb{R}} \frac{-D  \tilde{\phi}_1'' + \mu \tilde{\phi}_1 - \nu \tilde{\psi}_1}{\tilde{\phi}_1} , \inf_{\R \times [0,R)} \frac{-d \tilde{\psi}_1'' -f'(0) \psi_1}{\tilde{\psi}_1} \right\} \\
&\hspace{2cm}\phantom{\geq} + (1-t) \min \left\{ \inf_{\mathbb{R}} \frac{-D  \tilde{\phi}_2'' + \mu \tilde{\phi}_2 - \nu \tilde{\psi}_2}{\tilde{\phi}_2} , \inf_{\R \times [0,R)} \frac{-d \tilde{\psi}_2'' - f'(0) \tilde{\psi}_2}{\tilde{\psi}_2} \right\}.
\end{align*}
As $(\tilde{\phi}_1,\tilde{\psi}_1)$ and $(\tilde{\phi}_2,\tilde{\psi}_2)$ were chosen arbitrarily in respectively $\mathcal{E}_{\alpha_1}$, $\mathcal{E}_{\alpha_2}$, it follows that $\Lambda_R (\alpha) \geq t \Lambda_R (\alpha_1) + (1-t) \Lambda_R (\alpha_2)$. This concludes the proof.
\end{proof}
%
%
%
%
\section{Spreading speed of solutions}
\label{section:spreading}
We are now in a position to prove Theorem~\ref{main:spread1}. We begin by a characterization of the spreading speed $c^* (D)$ thanks to the generalized principal eigenvalue $\Lambda (\alpha)$ that we constructed in the previous section (see Theorem~\ref{th:eigenvalue} above).

Indeed by analogy with the single equation we introduce 
\begin{equation}\label{critical_speed}
c^* (D) := \min_{\alpha >0} \frac{-\Lambda (\alpha)}{\alpha}\in [c^*_{KPP},+\infty),
\end{equation}
where $c^*_{KPP} = 2\sqrt{df'(0)}>0$. This $c^* (D)$ is well-defined thanks to the following inequalities
\begin{equation}\label{eigen_bounds_whole}
\max \{ D\alpha^2 - \lambda_\alpha \;  , \; d\alpha^2 + f'(0) \} \leq -\Lambda (\alpha) \leq M_\alpha,
\end{equation}
which are in turn immediate consequences of Proposition~\ref{prop:monotonicity_wrt_R} and~\eqref{eigen_def}. Indeed note from \eqref{eigen_bounds_whole} that $-\Lambda$ has quadratic growth as $\alpha \to +\infty$ and $-\Lambda (\alpha)\geq f'(0)>0$, hence by continuity $\frac{-\Lambda (\alpha)}{\alpha}$ reaches its minimum as in \eqref{critical_speed}. Using again \eqref{eigen_bounds_whole}, it is also clear that $c^*(D)\geq \min \limits_{\alpha>0}\frac{d\alpha^2+f'(0)}{\alpha}=c^*_{KPP}$.

In particular the equation
$$
\Lambda (\alpha) = - c \alpha
$$
admits a solution $\alpha >0$ if and only if $c \geq c^* (D)$. Moreover, by the concavity of $\Lambda (\alpha)$ (Proposition~\ref{concavity}), if $c > c^* (D)$ then there are exactly two positive solutions. We can now establish the upper bound in our propagation result:
\begin{proof}[Proof of the first part of Theorem~\ref{main:spread1}]
For any solution $(u,v)$ of \eqref{eq:model}-\eqref{eq:initial_data} with non-negative and continuous compactly supported initial data we need to show that for any $c > c^* (D)$ and $R>0$ it holds
$$
\lim_{t \to +\infty}\   \sup_{x \leq -ct\; ,\; 0 \leq y \leq R} ( u (t,x) + v(t,x,y)) = 0.
$$
Recalling the discussion at the beginning of Section~\ref{sec:linear_pb}, $c\geq c^*(D)$ is a necessary and sufficient condition for existence of a positive solution of the form~\eqref{eq:exponential_supersolutions} to the linearized problem \eqref{eq:model_lin}. Thanks to the KPP assumption $f(u)\leq f'(0)u$ any such solution is also a supersolution of the original nonlinear problem \eqref{eq:model}.

More precisely, for any $c>c^*(D)$ choose some $c' \in [c^* (D),c)$ and $\alpha$ such that $-\Lambda (\alpha) = \alpha c'$, and let $(U_\alpha, V_\alpha)$ be the associated positive generalized eigenfunctions from Theorem~\ref{th:eigenvalue}. The pair $e^{\alpha (x+c't)} (U_\alpha,V_\alpha)$ lies, up to multiplication by some constant, above the initial datum $(u_0,v_0)$ at time $t=0$. Recalling that the solution $(u,v)$ is uniformly bounded, we can apply Proposition~\ref{prop:comparison1} to get
$$
\sup_{x \leq -ct\; ,\; 0 \leq y \leq R} ( u (t,x) + v(t,x,y)) \leq e^{\alpha (c' - c) t} \left( \|U_\alpha\|_{L^\infty (\mathbb{T})} + \| V_\alpha \|_{L^\infty (\Omega_R)}\right).
$$
and the desired conclusion immediately follows.
\end{proof}

From now on, this section will be dedicated to the proof of the inner spreading estimate, namely the fact that for any $0 < c < c^* (D)$ and $R>0$,
$$\lim_{t \to +\infty}\   \sup_{- ct \leq x \leq 0 \; ,\; 0 \leq y \leq R} ( |u (t,x) - U(x)| + |v(t,x,y) - V(x,y)|) = 0,$$
where $(U,V)$ is the unique positive and bounded stationary solution of \eqref{eq:model}.

\subsection{The construction of subsolutions}\label{sec:subsub}
In order to construct suitable subsolutions and obtain a lower estimate in the propagation Theorem~\ref{main:spread1} we first need the following technical result:
\begin{prop}\label{prop:analyticity}
For fixed $R>0$ the maps
$$
\alpha\mapsto\Lambda_R (\alpha)\in \mathbb{C}
\quad\mbox{and}\quad
\alpha\mapsto(U_{\alpha,R},V_{\alpha,R})\in \mathbb{C}^2
$$ can be holomorphically extended to a complex neighborhood of the positive real axis $\{\alpha>0\}$. The complex-valued eigenfunctions still satisfy \eqref{eq:truncated_eigen} together with some normalization, and moreover $\alpha\mapsto U_{\alpha,R},V_{\alpha,R}$ are
continuous with respect to the $C^{1,r}(\mathbb{T },\R^2)\times C^{1,r}(\overline{\Omega_R} , \R^2)$ topology (here we mean the usual topology for the real and imaginary parts of each component $U,V$ after identifying $\mathbb{C}^2\cong\R^2\times\R^2$).
\end{prop}
Note that the restriction to truncated domains is important here. Indeed, the generalized principal eigenvalue $\Lambda (\alpha)$ in the infinite cylinder is not an analytic function of $\alpha$, see the comment after Proposition~\ref{prop:eigen_enhance} later on.
%
 Since the equations are polynomial in~$\alpha$ and the domain is bounded our statement follows by standard perturbation theory combined with usual elliptic regularity and compactness arguments. We omit the details and refer to  \cite[Chapter 7]{Kato-perturbation}, observing that the principal eigenvalue $\Lambda_R (\alpha)$ is isolated and has algebraic and geometric multiplicity 1, thanks to the uniqueness of the positive eigenfunction (see Proposition~\ref{prop_KR} above).\\


For large $R>0$ we define
$$
c^*_R = \min_{\alpha >0} \frac{-\Lambda_R (\alpha)}{\alpha} \in (0, +\infty),
$$
which is the critical speed of the linearized problem in the truncated domain. Namely, system~\eqref{eq:model_lin} (restricted to $\Omega_R$) admits solutions of the exponential type \eqref{eq:exponential_supersolutions} if and only if $c \geq c^*_R$. For the construction of subsolutions we shall need the second technical result below:


%
\begin{lem}
Let $\alpha^*_R$ be the unique real and positive solution of $-\Lambda_R (\alpha_R^*) = c_R^* \alpha_R^*$.
There exist some small $r>0$ and $\delta>0$ such that, for any $c \in [c^*_R -\delta ,c^*_R)$, there exists a solution $\alpha (c)  \in \mathbb{C } \setminus \R$ of
$$
-\Lambda_R(\alpha)=\alpha c ,
$$
which also satisfies $| \alpha^*_R - \alpha (c) | \leq r$.
\label{lem:Rouche}
\end{lem}
Note that the existence and uniqueness of the positive solution of $-\Lambda_R (\alpha) = \alpha c^*_R$ follows from the definition of $c^*_R$, as well as from the concavity and analyticity of $\alpha \mapsto -\Lambda_R (\alpha)$. Then the conclusion of the lemma is an easy consequence of Rouch\'e's theorem and we omit the proof.\\

We can now construct solutions of the linearized problem in some moving sets with speed $c < c^*_R$:
\begin{prop}
For all $c \in [c^*_R - \delta ,c^*_R)$, there exist real valued functions $u_1 (t,x)$, $v_1 (t,x,y)$ and $h>0$ such that $(u_1,v_1)$ is a solution of \eqref{eq:model_lin}  in the moving set
$$\{ (t,x,y) \ : \quad t >0 , \ |x+ct | \leq h \mbox{ and } y \in (0,R) \},$$
and satisfies
\begin{equation}\label{cond_12}
\inf_{t >0} u_1 (t,-ct) >0 \ , \qquad \inf_{t>0} v_1 (t,-ct,R/2) >0,
\end{equation}
and
\begin{equation}\label{cond_11}
\left\{
\begin{array}{ll}
u_1 (t,x) \leq 0 & \mbox{ if } x =-ct \pm h , \vspace{3pt}\\
v_1 (t,x,y) \leq 0 & \mbox{ if either } x =-ct \pm h \mbox{ or } y =R .
\end{array}
\right.
\end{equation}
\end{prop}
\begin{proof}
%
%
%
%
Let $c \in [c^*_R - \delta , c^*_R)$ and, by Lemma~\ref{lem:Rouche} there exists a solution
$$
\alpha \in \mathbb{C} \setminus\mathbb{R}:\qquad -\Lambda_R(\alpha)=\alpha c.
$$
By Proposition~\ref{prop:analyticity}, there is a corresponding complex-valued eigenfunction pair $(U_R(x),V_R(x,y))$ of \eqref{eq:truncated_eigen}. By construction the functions
$$
(u_1 (t,x),v_1(t,x,y))=\re\left(e^{\alpha(x+ct)}\left(U_R(x),V_R(x,y)\right)\right)
$$
are real valued solutions to \eqref{eq:model_lin} on the truncated domain $\{(t,x,y) : \ 0 \leq y \leq R\}$, together with the Dirichlet boundary condition $V_R (\cdot,y=R) \equiv 0.$

Let us check that conditions \eqref{cond_12} and \eqref{cond_11} are satisfied. Note first that, as $c \to c^*_R$, then $\alpha \to \alpha^*_R$ and, by the continuity of $(U_{\alpha,R},V_{\alpha,R})$ with respect to $\alpha$ (see Proposition~\ref{prop:analyticity}), we can assume up to reducing $\delta$ and without loss of generality that
\begin{equation}
\label{subsol_claim1}\min_{x \in \mathbb{T}} \re (U_R (x)) > \frac{1}{2} \min_{x \in \mathbb{T}} U^*_R(x) >0 \, , \qquad \min_{x \in \mathbb{T} , 0  \leq y \leq R} \re (V_R (x,y)) \geq 0,
\end{equation}
where $(U^*_R ,V^*_R)$ denotes the (normalized) principal eigenfunction pair of~\eqref{eq:truncated_eigen} with $\alpha= \alpha^*_R$.

Indeed for $U_R$ the argument is straightforward, as well as for $V_R$ away from the upper boundary $y=R$. More precisely, it is clear that for any $0 < \varepsilon <R$, then we can decrease $\delta$ so that $\alpha$ is sufficiently close to $\alpha^*_R$ and
$$
\min_{x \in \mathbb{T} , 0 \leq y \leq R-\varepsilon} \re (V_R (x,y)) > \frac{1}{2} \min_{x \in \mathbb{T} , 0 \leq y \leq R-\varepsilon} V^*_R (x,y) > 0.
$$
In particular, \eqref{cond_12} is satisfied.

In order to deal with $y\geq R-\varepsilon$ in \eqref{subsol_claim1}, recall that the real valued function $V^*_R>0$ attains its minimum at any boundary point $y=R$, hence by Hopf lemma and periodicity, $\partial_y V^*_R|_{y=R}\leq -c_0<0$. Using again Proposition~\ref{prop:analyticity} and in particular the continuity with respect to $\alpha$ in the $C^{1,r}(\overline{\Omega_R})$ topology (up to $y=R$), we see that there holds $ \partial_y \re (V_R)|_{y=R}\leq -c_0/2<0$ up to reducing $\delta$ again. Since $V_R|_{y=R}=0$ we get $\re (V_R)>0$ for $y\in [R-\eps,R)$, whence \eqref{subsol_claim1}.

It then immediately follows that, letting $\rho=\re \, \alpha >0$, $\omega=\im  \, \alpha\neq 0$, and $h=\pi/|\omega|$:
\begin{eqnarray*}
\left( u_1 (t,-ct \pm h), v_1 (t,-ct \pm h,y) \right)&=& e^{\pm\rho h}\re\left(e^{\pm i\pi}\left(U_R(-ct\pm h),V_R(-ct\pm h,y)\right)\right)\\
& \leq  & (0,0).
\end{eqnarray*}
Recalling that $v_1$ satisfies the Dirichlet boundary condition at $y=R$, this ends the proof of the proposition.
\end{proof}
It is straightforward to check, by the uniqueness of the principal eigenvalue, that the function $\Lambda_R (\alpha)$ depends continuously on the parameter $f'(0)$, in the $L^\infty_{loc} ([0,+\infty ))$ topology. Thus, the speed $c^*_R$ also depends continuously on the parameter $f'(0)$. In particular, for any small $\varepsilon >0$, the proposition above still holds true if $c^*_R$ is replaced by some $c^*_{R,\varepsilon}$ which converges to $c^*_R$ as $\varepsilon \to 0$, and $f'(0)$ is replaced by $f'(0) - \varepsilon$ in~\eqref{eq:model_lin}.

Denoting by $u_{1,\varepsilon}$ and $v_{1,\varepsilon}$ the functions given by the above proposition applied to this perturbed problem, let then for any $t>0$, $|x+ct| \leq h$ and $0 \leq y \leq R$:
$$
\underline{u}(t,x)=
\max (u_{1,\eps}(t,x),0)	
 \ , \qquad
\underline{v}(t,x,y)=
\max (v_{1,\eps} (t,x,y),0)	.
$$
Then, by the KPP assumption and more specifically the regularity of $f$ in a neighborhood of 0, the pair $(\underline{u},\underline{v})$ is, up to multiplication by some small constant, a generalized subsolution of the original nonlinear problem \eqref{eq:model} in a compactly supported set moving with speed 
$c \in [c^*_{R,\varepsilon} - \delta , c^*_{R,\varepsilon})$. Moreover, choosing $R$ large and $\varepsilon$ small enough, the speed~$c$ can clearly be chosen arbitrarily close to $c^* (D)$.

More precisely, we have proved the following proposition:
\begin{prop}\label{prop:subsol_nonlinear}
There exist $c \in [0,c^*)$ arbitrarily close to $c^*$, and a pair $(\underline{u},\underline{v})$ such that, for any $0 \leq \kappa \leq 1$, $\kappa (\underline{u},\underline{v})$ is a continuous, non-negative and bounded (generalized) subsolution of \eqref{eq:model} in a domain $E \times F$ where 
$$
E:=\{ (t,x) \ : \quad -ct-h < x < -ct+h\}\subset (0,\infty)\times \R,
$$
$$F:= E \times (0,R).
$$
Moreover, $\underline{u} (t,x+ct)$ and $\underline{v} (t,x+ct,y)$ are $\frac{L}{c}$-periodic with respect to time, and satisfy that\begin{equation}\label{cond_12b}
\inf_{t >0} \underline{u} (t,-ct) >0 \ , \qquad \inf_{t>0} \underline{v} (t,-ct,R/2) >0,
\end{equation}
and
\begin{equation}\label{cond_11b}
\left\{
\begin{array}{ll}
\underline{u} (t,x) = 0 & \mbox{ if } x =-ct \pm h , \vspace{3pt}\\
\underline{v} (t,x,y) = 0 & \mbox{ if either } x =-ct \pm h \mbox{ or } y =R .
\end{array}
\right.
\end{equation} 
\end{prop}
Time periodicity and conditions \eqref{cond_12b}, \eqref{cond_11b} are immediate consequences of the construction of $\underline{u}$, $\underline{v}$ above. While \eqref{cond_12b} ensures that the subsolution is not trivial, \eqref{cond_11b} is required in order to apply a comparison principle (after extending $(\underline{u},\underline{v})$ by $(0,0)$ outside of the strip $|x+ct|\leq h$).
%
\subsection{Proof of the inner spreading theorem}\label{sec:inner}
We can finally prove the inner spreading part of Theorem~\ref{main:spread1}. Recall that $(U,V)$ denotes the unique positive and bounded stationary solution of \eqref{eq:model}. Let us first prove the following two lemmas:
\begin{lem}\label{lem:above_spread}
Let $(u,v)$ be a solution of \eqref{eq:model} with bounded initial data $(u_0,v_0)$. Then
$$\limsup_{t \to \infty} \ \sup_{x \in \R} \  u (t,x) - U (x) \leq 0,$$
$$\limsup_{t \to \infty} \sup_{x \in \R, y \geq 0 } v(t,x,y) - V(x,y) \leq 0. $$
\end{lem}
\begin{proof}
We begin by noting that, thanks to the KPP hypothesis, the pair $\gamma (U,V)$ is a supersolution of \eqref{eq:model} for any real number $\gamma >1$. In particular, the solution of the Cauchy problem \eqref{eq:model} with initial data $(\gamma U , \gamma V)$ is non-increasing with respect to time. Moreover, by Proposition~\ref{prop:comparison1}, it is bounded from below by the stationary solution $(U,V)$. Therefore, it must converge as $t \to +\infty$ to a bounded positive stationary solution, which must be $(U,V)$ itself. 

On the other hand, as $u_0$ and $v_0$ are bounded and $U$, $V$ have positive infimum, there exists some $\gamma >1$ such that $(u_0 ,v_0) < \gamma (U,V)$. Applying again Proposition~\ref{prop:comparison1}, we get the wanted conclusion.
\end{proof}

\begin{lem}\label{lem:local_spread}
Let $(u,v)$ be a solution of \eqref{eq:model} with bounded, non-negative and non-trivial initial data $(u_0,v_0)$. Then $(u,v)$ converges locally uniformly to $(U,V)$ as $t \to +\infty$.
\end{lem}
\begin{proof}
The previous lemma already proved that the supremum limits of $u$ and $v$ lie below $U$ and $V$. Thus, we only have to prove that for any $K >0$, then
$$\liminf_{t \to \infty} \inf_{|x| \leq K , 0 \leq y \leq K} u (t,x) - U (x) \geq 0,$$
$$\liminf_{t \to \infty} \inf_{|x| \leq K , 0 \leq y \leq K } v(t,x,y) - V(x,y) \geq 0. $$
Recall that in a previous section, we have constructed arbitrarily small and compactly supported stationary subsolution pairs $(\underline{U} (x), \underline{V} (x,y))$ such that the associated solution of the Cauchy problem~\eqref{eq:model} converges locally uniformly to the unique bounded and positive stationary solution $(U,V)$ (see the proof of Lemma~\ref{lem:liouville_exist}). 

Moreover, as our system satisfies a strong parabolic maximum principle (see Proposition~\ref{prop:comparison1}), the solution $(u,v)$ is positive everywhere for any positive time and we can assume without loss of generality that 
$$u(t=1,\cdot) \geq \underline{U} (\cdot) , \quad v(t=1,\cdot) \geq \underline{V} (\cdot).$$
Applying again the comparison principle, the conclusion follows.
\end{proof}

We now go back to the proof of Theorem~\ref{main:spread1}. Thanks to the two lemmas above, it only remains to prove that, for any $0 < c < c^* (D)$ and $R>0$, then
$$\liminf_{t \to \infty} \inf_{-1- c^* (D) \geq x  \geq -ct} u (t,x) - U (x) \geq 0,$$
$$\liminf_{t \to \infty} \inf_{-1- c^* (D) \geq x \geq -ct , 0 \leq y < R } v(t,x,y) - V(x,y) \geq 0. $$
From now on, $c \in (0, c^* (D))$ is fixed. Then, let $c < c' < c^* (D)$ close to $c^* (D)$ so that,  by Proposition~\ref{prop:subsol_nonlinear}, there exists a moving subsolution $(\underline{u}, \underline{v})$ of \eqref{eq:model} with speed~$c'$. Again, Proposition~\ref{prop:comparison1} implies that, for any $t>0$, the solution $(u,v)$ is positive in the whole domain. Thus, up to multiplication of $(\underline{u}, \underline{v})$ by some small enough $\kappa >0$, we can assume without loss of generality that for any $x \in E$ and $y \in F$:
$$u(1,x) \geq \underline{u} (1,x) \ , \qquad v (1,x,y) \geq \underline{v} (1,x,y).$$
Furthermore, we also have that for any $(t,x) \in \partial E \cap ((1,\infty) \times \R )$,
$$u(t,x) \geq 0 = \underline{u} (t,x),$$
and for any $(t,x,y) \in \partial F \cap  ((1,\infty) \times \Omega) $
$$v(t,x) \geq 0 = \underline{v} (t,x,y).$$
Applying Proposition~\ref{prop:comparison2} (note that $\overline{F} \cap \{y=0\} = \overline{E}$), one may immediately conclude that for any $t \geq 1$ and $(x,y) \in F$:
$$u (t,x) \geq \underline{u} (t,x) \ , \qquad v(t,x) \geq \underline{v} (t,x).$$

Now proceed by contradiction and assume that there exist $\delta >0$ and some sequences $t_n \to +\infty$, $x_n \in [-1-c^* (D),-ct_n]$ and $y_n \in [0,R]$, such that either $u (t_n,x_n) < U(x_n) -\delta$ or $v(t_n,x_n,y_n) < V (x_n,y_n) - \delta$. Note that, by parabolic estimates~\cite{ LadyzhenskayaUraltseva}, the sequences $u(t_n+s,x_n+x)$ and $v(t_n+s,x_n+x,y)$ converge locally uniformly to a solution $(u_\infty,v_\infty)$ of \eqref{eq:model}, which by Lemma~\ref{lem:above_spread} lies below $(U,V)$. From the choice of $t_n$ and $x_n$ and a strong maximum principle argument, it is straightforward to check that $(u_\infty,v_\infty)$ even lies strictly below $(U,V)$. Thus, the sequence $x_n$ can be replaced by any sequence of points $x'_n$ such that $x'_n - x_n$ is bounded. In particular, we can assume without loss of generality that $x_n = k_n L$ with $k_n \in \mathbb{Z}$ for any $n \in \mathbb{N}$.

We can now let $t'_n = \frac{|x_n|}{c'} >1$. Then
$$u (t'_n,x_n+x) \geq \underline{u} (t'_n,x_n + x) = \underline{u} (0,x),$$
$$v (t'_n,x_n+x,y) \geq \underline{v} (t'_n,x_n + x,y) = \underline{v} (0,x,y),$$
where the equalities follow from the time periodicity of the subsolution in the moving frame with speed $c'$.

%

From Lemma~\ref{lem:local_spread}, the solution of \eqref{eq:model} with initial data $(\underline{u} (0,x), \underline{v} (0,x,y))$ converges locally 
uniformly to $(U,V)$. Therefore, applying the comparison principle and noting that $t_n - t'_n \to +\infty$ as $n \to +\infty$, 
we get that $(u_\infty,v_\infty) \geq (U,V)$ and reach a contradiction. This ends the proof of Theorem~\ref{main:spread1}.

\section{Speed enhancement by the road}\label{sec:accelerate}

We have now computed the spreading speed $c^* (D)$ of solutions of \eqref{eq:model} with compactly supported initial data, in the sense of Theorem~\ref{main:spread1}. In this section, we will compare $c^* ( D)$ with the spreading speed $c^*_{KPP}$ of solutions of the single KPP equation. In particular, we will prove Theorem~\ref{main:speedup} and show that the road accelerates the propagation (for compactly supported initial data) if and only if $D>2d$.

The proof will rely on Proposition~\ref{prop:eigen_enhance} below, which compares $\Lambda (\alpha)$ to the principal eigenvalue arising when looking for exponential solutions of the homogeneous KPP problem with no road. Furthermore, we will prove in the last subsection that $\Lambda (\alpha)$ also characterizes the spreading speed of solutions with exponentially decaying initial data. Therefore, Proposition~\ref{prop:eigen_enhance} also infers whether such solutions are accelerated by the road or not.

%
\subsection{Eigenvalue enhancement by the road}
Noting that 
$$c^*_{KPP} = \min_{\alpha >0} \frac{d\alpha^2 + f'(0)}{\alpha} = \frac{d\alpha_{KPP}^2 + f'(0)}{\alpha_{KPP}} $$
with $\alpha_{KPP} = \sqrt{\frac{f'(0)}{d}}$, the fact that $c^* (D) > c^*_{KPP}$ if and only if $D>2d$ is a simple corollary of the following proposition:
\begin{prop}\label{prop:eigen_enhance}~
\begin{enumerate}[$(i)$]
\item If $D \leq d$ then
$$
\forall \,  \alpha \geq 0:  \qquad  -\Lambda (\alpha) = d\alpha^2 + f'(0) .$$
\item If $D > d$ and $\alpha (D) = \sqrt{\frac{f'(0)}{D-d}}$ then
\begin{equation*}
\begin{array}{lcl}
\forall  \, 0\leq\alpha \leq \alpha (D): & \qquad & -\Lambda (\alpha) = d\alpha^2 + f'(0) , \vspace{3pt}\\
\forall \,  \alpha > \alpha (D): & \qquad & - \Lambda (\alpha) > d\alpha^2 + f'(0).
\end{array}
\end{equation*}
\end{enumerate}
\end{prop}
From $(ii)$ it is now clear that the generalized principal eigenvalue $\Lambda(\alpha)$ in general cannot be analytical in $\alpha$, which contrasts with Proposition~\ref{prop:analyticity} in the case of truncated cylinders.

\begin{proof}
Recalling from \eqref{eigen_bounds_whole} that $-\Lambda (\alpha) \geq d\alpha^2 + f'(0)$ we have for all $\alpha\geq 0$ that
$$
\omega= \sqrt{\frac{-d\alpha^2 - f'(0) - \Lambda (\alpha)}{d}}\geq 0,
$$
and that our statement amounts to determine whether $\omega=0$ or $\omega>0$. Letting $U$ and $V$ be the positive eigenfunctions of \eqref{eq:principal_eigenvalue} from Theorem~\ref{th:eigenvalue}, integrating the field equation with respect to $x$, and denoting $\Phi (y) = \int_{x \in \mathbb{T}} V (x,y) dx$, it is easy to check as before that
$$
\Phi''(y)-\omega^2\Phi(y)=0.
$$
Since $V(x,y)>0$ grows at most linearly in $y$ (Theorem~\ref{th:eigenvalue}) this immediately implies that either $\Phi(y)=ay+b$ for some $a,b\geq 0$ if $\omega=0$, or $\Phi(y)=Ce^{-\omega y}$ for some $C>0$ if $\omega>0$.
Integrating now the boundary condition and the road equation, an explicit computation gives
\begin{equation}\label{eq:speedup}
-d\Phi'(0)=\int_{\mathbb{T}} \{\mu U-\nu V(\cdot ,0)\}\mathrm{d}x=[D\alpha^2+\Lambda(\alpha)]\int_{\mathbb{T}} U\mathrm{d}x.
\end{equation}
If either $D\leq d$, or $D>d$ and $0\leq \alpha\leq\alpha(D)$, then $D \alpha^2 \leq d\alpha^2 + f'(0)$. From the above equality, we get
$$- d \Phi ' (0) \leq  [ d \alpha^2 + f'(0)  +\Lambda (\alpha) ] \int_{\mathbb{T}} U = - d \omega^2 \int_{\mathbb{T}} U.$$
Thus $\Phi'(0)\geq \omega^2 \int_{\mathbb{T}} U \geq 0$. It follows that $\omega=0$ (otherwise $\Phi(y)$ would be an exponential with slope $-C\omega<0$ at $y=0$), which is exactly the desired conclusion in those two cases.

In the last case $D>d$ and $\alpha>\alpha(D)$, then $D \alpha^2 > d\alpha^2 + f'(0)$ which, together with \eqref{eq:speedup}, leads to
$$\Phi '(0) < \omega^2 \int_{\mathbb{T}} U.$$
As $\Phi' (0) \geq 0$ when $\omega = 0$, it follows that $\omega >0$ and the proof is achieved.
\end{proof}

%
\subsection{Large diffusion limit}

Let us now end the proof of Theorem~\ref{main:speedup} by checking that
$$
0 < \liminf_{D \to \infty} \frac{ c^*(D)}{\sqrt{D}} \leq \limsup_{D \to \infty} \frac{ c^*(D)}{\sqrt{D}} < +\infty.
$$
In order to obtain the upper bound we take first $\alpha=\alpha(D)=\sqrt{\frac{f'(0)}{D-d}}$ for $D>d$, and compute explicitly
$$
c:=\frac{d\alpha^2+f'(0)}{\alpha}=D\sqrt{\frac{f'(0)}{D-d}}.
$$
Because we chose exactly $\alpha=\alpha(D)$ we are in case $(ii)$ of Proposition~\ref{prop:eigen_enhance} with $-\Lambda(\alpha)=d\alpha^2+f'(0)$, thus we just exhibited a solution $\alpha=\alpha(D)$ of $-\Lambda(\alpha)=\alpha c$ for this particular value of $c$. By definition of $c^* (D)$ this means that $c^*(D)\leq c$, and it immediately follows that
$$ \limsup_{D \to \infty} \frac{ c^*(D)}{\sqrt{D}}  \leq \sqrt{f'(0)} < +\infty .$$

Turning now to the lower bound, recall from \eqref{eigen_bounds_whole} that $-\Lambda(\alpha)\geq \max\{D\alpha^2-\lambda_\alpha,d\alpha^2+f'(0)\}$, where $\lambda_\alpha \in [\mu_0,\mu_1]$ is the principal eigenvalue of $-D\frac{d^2}{dx^2}-2D\alpha\frac{d}{dx}+\mu(x)$ on the torus. In particular we have
$$
-\Lambda(\alpha)\geq p(\alpha):=\max\{D\alpha^2-\mu_1,d\alpha^2+f'(0)\}.
$$
Studying the piecewise polynomial and convex function $p(\alpha)$ for fixed $D>d$ it is easy to check that $\min\limits_{\alpha > 0}\frac{p(\alpha)}{\alpha}\geq \frac{f'(0)}{\sqrt{f'(0)+\mu_1}}\sqrt{D-d}$. Since $c^*(D)=\min\limits_{\alpha > 0}\frac{-\Lambda(\alpha)}{\alpha}\geq \min\limits_{\alpha > 0}\frac{p(\alpha)}{\alpha}$, it immediately follows that
$$ \liminf_{D \to \infty} \frac{ c^*(D)}{\sqrt{D}}  \geq \frac{f'(0)}{\sqrt{f'(0)+\mu_1}}>0.$$
%

\subsection{Exponentially decaying initial data}\label{sec:exp_case}

Let us now turn to the proof of Proposition~\ref{exp_decay_case}. By Proposition~\ref{prop:eigen_enhance}, we need to prove that $c (\alpha) = \frac{-\Lambda (\alpha)}{\alpha}$ is the spreading speed of solutions for initial data with exponential decay of order $e^{\alpha x}$. This should be natural by now since for any fixed $\alpha$ and by construction of $\Lambda (\alpha)$ (see again Theorem~\ref{th:eigenvalue}), $c (\alpha) $ is the smallest speed~$c$ such that there exists a positive solution of the linearized problem~\eqref{eq:model_lin} of the type~\eqref{eq:exponential_supersolutions}.

\begin{proof}[Proof of Proposition~\ref{exp_decay_case}]
Let us first check the upper estimate. Note that although $e^{\alpha (x +c(\alpha)t)} (U_\alpha (x), V_\alpha (x,y))$ is an obvious supersolution of~\eqref{eq:model}, the difficulty lies in the fact that $e^{\alpha x} V_\alpha (x,y)$ may decay as $y \to +\infty$ and, therefore, it may not lie above the initial datum $v_0$.

Choose any $c > c(\alpha)$ and  observe that $c \alpha > -\Lambda (\alpha)$ by definition of $c(\alpha)=-\frac{\Lambda(\alpha)}{\alpha}$. An explicit computation shows that
$$
(\overline{U}_1,\overline{V}_1) := e^{\alpha (x + c t)} (U_\alpha , V_\alpha)
$$
satisfies
$$\left\{ 
\begin{array}{ll}
\partial_t \overline{U}_1 - D \partial_x^2 \overline{U}_1 + \mu \overline{U}_1 - \nu \overline{V}_1|_{y=0} = (\Lambda (\alpha) + c \alpha) \overline{U}_1  , & t \geq 0, x \in \R,\vspace{3pt}\\
\partial_t \overline{V}_1  - d \Delta \overline{V}_1 - f' (0) \overline{V}_1 = (\Lambda (\alpha) + c \alpha) \overline{V}_1  , & t\geq 0, (x,y) \in \Omega ,\vspace{3pt}\\
-d \partial_y \overline{V}_1|_{y=0} = \mu \overline{U}_1 - \nu \overline{V}_1|_{y=0} ,  & t \geq 0 , x \in \R.
\end{array}
\right.
$$
Also defining
$$
(\overline{U}_2,\overline{V}_2) := (0, e^{\alpha (x + c t)} )
$$
and recalling that $-\Lambda (\alpha) \geq d\alpha^2 + f'(0)$, it is easy to check that
$$
\left\{
\begin{array}{ll}
\partial_t \overline{V}_2  - d \Delta \overline{V}_2 - f' (0) \overline{V}_2 \geq (\Lambda (\alpha) + c \alpha) \overline{V}_2  , & t\geq 0, (x,y) \in \Omega ,\vspace{3pt}\\
-d \partial_y \overline{V}_2|_{y=0} \geq \mu \overline{U}_2 - \nu \overline{V}_2|_{y=0}  , & t \geq 0 , x \in \R.
\end{array}
\right.
$$
Then, letting $M$ large enough so that 
$$
M(\Lambda (\alpha) + c\alpha) \min U_\alpha \geq \nu_1,
$$
it is straightforward to check that $M(\overline{U}_1,\overline{V}_1) + (\overline{U}_2,\overline{V}_2)$ is a supersolution of the linearized problem~\eqref{eq:model_lin} and is moving with speed $c>c(\alpha)$. Up to multiplication by some large positive constant and from our assumptions in Proposition~\ref{exp_decay_case}, it can be assumed to lie above $(u_0,v_0)$ at time $t=0$. It is then straightforward, applying Proposition~\ref{prop:comparison1} and noting that $c$ could be chosen arbitrarily close to $c(\alpha)$, to reach the desired conclusion.
\medskip

Let us now focus on the inner spreading estimate. Proceeding as in the proof of Theorem~\ref{main:spread1} (see Section~\ref{sec:inner}), it is enough to find a non-trivial subsolution lying below $(u_0,v_0)$ at time $t=0$ and moving with speed $c$ smaller but arbitrarily close to $c (\alpha)$. 
%
%
Let any $R>0$ and define
$$
c_R (\alpha) = \frac{-\Lambda_R (\alpha)}{\alpha} \nearrow c (\alpha) \mbox{ as } R \to +\infty .
$$
In particular $c_R (\alpha) > c^*_R$ provided that $R$ is large enough, and, by concavity of $\Lambda_R (\alpha)$, one has
$$c_R (\alpha) (\alpha + \eta) > -\Lambda_R (\alpha + \eta)$$ for $\eta>0$ small enough. For simplicity we write below $c=c_R(\alpha)$,
and define
$$
\underline{U}_1 =  A e^{\alpha (x+c t)} U_{\alpha,R} - B e^{(\alpha+\eta) (x+ c t)} U_{\alpha + \eta,R} ,
$$
$$
\underline{V}_1 = 
 A e^{\alpha (x+c  t)} V_{\alpha,R} - B e^{(\alpha+\eta) (x+c  t)} V_{\alpha + \eta,R}   \quad \text{if } 0 \leq y \leq R,
 $$
where $A$ and $B$ are positive constants to be chosen later.

One can check that for all $t \geq 0$ and $x \in \R$ there holds
\begin{align*}
&\partial_t \underline{U}_1 - D \partial_x^2 \underline{U}_1 + \mu \underline{U}_1 - \nu \underline{V}_1|_ {y=0} \vspace{3pt}\\
&\hspace{2cm}=A\cdot 0  -B \big[\Lambda_R  (\alpha + \eta) + ( \alpha+\eta)c \big] e^{(\alpha+\eta) (x + c t)} U_{\alpha+\eta,R}< 0,
\end{align*}
and
$$
-d \partial_y \underline{V}_1|_{y=0} - \mu \underline{U}_1 + \nu \underline{V}_1|_{y=0} = 0 = \underline{V}_1|_{y=R}.
$$
By $C^{1,r}$ regularity of $f$ there exist a small $s_0>0$ and some $\gamma >0$ such that 
$$
\forall v \in (-\infty, s_0 ]:\qquad  f (v) \geq f'(0) v - \max \{ 0 , \gamma v \}^{1+r}.
$$
Note that up to now $f(v)$ was only defined for $v \geq 0$, hence we may indeed assume without loss of generality that $f(v)=f'(0) v$ for $v \leq 0$ and, in particular, the above inequality holds. 

Choose first $A>0$ small so that $A V_{\alpha,R}\leq s_0<1$, and then $B>0$ large enough so that $\underline{U}_1 \leq 0$ and $\underline{V}_1 \leq 0$ for all $x + c t \geq 0$. If moreover $\eta>0$ is small enough such that $\alpha+\eta<\alpha(1+r)$ then for $x\leq -c t$ it is easy to check that
\begin{align*}
\max\{0, \gamma \underline{V}_1 \}^{1+r}&\leq \big( \gamma A e^{\alpha (x+c t)} V_{\alpha,R}\big)^{1+r}\\
&\leq  e^{\alpha (1+r) (x+c t)} ( \gamma A V_{\alpha,R})^{1+r} \leq e^{(\alpha+\eta) (x+c t)} ( \gamma A V_{\alpha,R})^{1+r} ,
\end{align*}
whence
\begin{align*}
&  \partial_t \underline{V}_1  - d \Delta \underline{V}_1 - f(\underline{V}_1)= \Big[f'(0)\underline{V}_1-f(\underline{V}_1)\Big] + \Big[ \partial_t \underline{V}_1  - d \Delta \underline{V}_1 - f'(0)\underline{V}_1\Big]\\
& \hspace{1cm}\leq
\max\{0, \gamma \underline{V}_1 \}^{1+r} + \Big[A\cdot 0 -B( \Lambda_R (\alpha + \eta)  + (\alpha +\eta)c)  e^{(\alpha+\eta) (x + ct)}V_{\alpha + \eta ,R}  \Big]\\
& \hspace{1cm}\leq e^{(\alpha+\eta) (x+c t)}\Big[ (\gamma A  V_{\alpha,R})^{1+r} -B( \Lambda_R (\alpha + \eta)  + (\alpha +\eta)c )  V_{\alpha + \eta ,R}\Big]<0
\end{align*}
if $B>0$ is large enough. On the other hand when $x \geq -c t$ the same inequality $\partial_t\underline{V}_1-d\Delta \underline{V}_1-f(\underline{V}_1)\leq 0$ easily follows from the fact that $\underline{V}_1 \leq 0$ and thus $f (\underline{V}_1) = f'(0) \underline{V}_1$.
Therefore the pair
$$
\underline{U}: = \max \{ 0, \underline{U}_1 \} ,\qquad \underline{V}:=\left\{
\begin{array}{ll}
\max \{ 0, \underline{V}_1 \} & \text{if } 0 \leq y \leq R,\\
0 & \text{elsewhere},
\end{array}
\right. 
$$
is a generalized subsolution of \eqref{eq:model} in the sense of Proposition~\ref{prop:comparison2}. By construction this subsolution moves with speed $c=c_R (\alpha)=-\frac{\Lambda_R(\alpha)}{\alpha}$, which can be made arbitrarily close to $c (\alpha)=-\frac{\Lambda(\alpha)}{\alpha}$ by choosing $R$ large enough. Moreover, with our assumption on the initial data and up to multiplication by a small constant, this subsolution can be assumed to lie below $(u_0,v_0)$ at time $t=0$ (observe that $\underline{U}=\underline{V}=0$ for $x\geq -c t$). As mentioned above, one may then proceed as in Section~\ref{sec:inner} to end the proof of Proposition~\ref{exp_decay_case}.
\end{proof}

%
%
\section{Some extensions}\label{sec:extension}

The crucial feature of the road-field system is that it is of a cooperative type, which in particular allowed us to construct and characterize the spreading speed through a family of principal eigenvalues and Krein-Rutman theory. We briefly discuss here some extensions to more general frameworks where a similar approach could be performed.

\paragraph{Truncated fields:} We first highlight here the fact that the generalized principal eigenvalues $\Lambda (\alpha)$ were obtained as the limits of the principal eigenvalues $\Lambda_R (\alpha)$ of some truncated field problems (see Section~\ref{sec:linear_pb} for the details). Therefore, it should be clear from our proofs that a similar result as Theorem~\ref{main:spread1} holds for the restriction of~\eqref{eq:model} to a strip domain $\{(x,y) \in \R^2 \, : \ 0<y <R\}$ with a Dirichlet boundary condition $v(y=R)=0$. 

More precisely, provided that $R$ is large enough, the solutions spread with the speed $c^*_R \in (0,\infty)$ along the road, where $c^*_R$ was defined in Section~\ref{section:spreading}. When $\mu$ and $\nu$ are positive constants, this was proved in~\cite{Tellini}, along with the fact that there exists a critical diffusion $D$ above which $c^*_R > c^*_{KPP}$ and below which $c^*_R < c^*_{KPP}$. This relied on the monotonicity of $c^*_R$ with respect to $D$, which remains an open problem in our periodic framework.

Noting that the propagation is slowed down by the Dirichlet condition on the upper boundary and in order to isolate the effect of the line of fast diffusion, one may also want to compare $c^*_R$ with
$$c^*_{KPP,R} := 2\sqrt{d f'(0) -  \frac{d^2 \pi^2}{4 R^2}} = \min_{\alpha >0} \frac{d\alpha^2 + f'(0) - d\frac{\pi^2}{4R^2}}{\alpha},$$
which one may check to be the critical spreading speed of solutions of
\begin{equation}\label{presque}
\left\{
\begin{array}{ll}
\partial_t v-d\Delta v=f(v) , & t>0, x\in \R , y \in (0,R),\\
\partial_y v (t,x,0) = 0 = v (t,x,R) , & t>0,x\in\R .
\end{array}
\right. 
\end{equation}
Because the main idea of the proof of Proposition~\ref{prop:eigen_enhance} (thus of Theorem~\ref{main:speedup} and Proposition~\ref{exp_decay_case}) roughly consists in determining the direction of the (average) flow between the road and the field, the above problem naturally arises. Note also that, in the limit case $R=+\infty$, it is clearly equivalent to \eqref{eq:KPP}.

Therefore, by a similar argument as in Section~\ref{sec:accelerate}, one may prove that 
$$-\Lambda_R (\alpha) < d\alpha^2 + f'(0) - d\frac{\pi ^2}{4R^2}$$ if 
$$\alpha < \alpha_R (D) := \sqrt{\frac{f'(0) - d \frac{\pi^2}{4R^2}}{D-d}},$$
while the opposite strict inequality holds if $\alpha > \alpha_R (D)$. 

Then, letting $\alpha^*_R$ be the unique (thanks to the analyticity of $-\Lambda_R$) positive solution of $c^*_R \alpha = -\Lambda_R (\alpha)$, it is straightforward to obtain the analogous of Proposition~\ref{exp_decay_case} in the truncated field case. Namely, solutions with exponential decay of order $e^{\alpha x}$, where $\alpha < \alpha_R^* $, are accelerated by the road in the sense that they spread with some speed 
$$c_R (\alpha) = \frac{-\Lambda_R (\alpha)}{\alpha} >  \frac{d\alpha^2 + f'(0) - d\frac{\pi^2}{4R^2}}{\alpha},$$
if and only if $\alpha > \alpha_R (D)$.  The right-hand term of the above inequality is, of course, the spreading speed of solutions of \eqref{presque} with decay of the same order $e^{\alpha x}$.

However, it is not clear in general whether the minimal speed $c^*_R$ is larger than $c^*_{KPP,R}$. This is related to the fact that we cannot always locate $\alpha^*_R$ where the minimum in the definition of $c^*_R$ is reached. In particular, the above discussion does not provide a precise diffusion threshold for the acceleration of solutions with compactly supported initial data (or more generally, initial data that decay faster than $e^{\alpha^*_R x}$). It is only straightforward to infer from the above that $c^*_R < c^*_{KPP,R}$ if $D < 2d$ as well as, using Proposition~\ref{prop:monotonicity_wrt_R} and more specifically the estimates on $-\Lambda_R (\alpha)$, that $c^*_R > c^*_{KPP,R}$ for large enough~$D$.

\paragraph{Including KPP reaction on the road:} The key assumption, which in particular we used extensively in the investigation of the effect of the road, is the fact that the heterogeneity is limited to the exchange terms. Therefore, one may easily check that our method still applies when replacing the road equation in
\eqref{eq:model} by
$$
\partial_t u-D\partial_x^2 u=\nu(x)v(t,x,0)-\mu(x) u + g(u),
$$
where $g(u)$ satisfies a KPP type assumption
$$g(0)=0, \; u \mapsto \frac{g(u)}{u} \mbox{ is non-increasing,} \; \mbox{ and } \frac{g(u)}{u} \leq 0 \mbox{ for large enough } u.$$
We leave the details of the construction of the spreading speed to the readers: the proof of Section~\ref{sec:linear_pb} applies to the letter up to minor changes (by simply adding $g'(0) U$ to the operator $L_1$, $g'(0)$ in the definition of $M_\alpha$, and denoting by $\lambda_\alpha$ the principal eigenvalue of $-D \frac{d^2}{dx^2} - 2 \alpha D \frac{d}{dx} + \mu - g'(0)$). For the sake of simplicity, let us just assume that we have constructed for this new problem a family of eigenvalues $-\Lambda_g (\alpha) \geq d\alpha^2 + f'(0)$ and a spreading speed
$$c^*_g (D) = \min_{\alpha >0} \frac{-\Lambda_g (\alpha)}{\alpha}.$$
We now focus on Proposition~\ref{prop:eigen_enhance}. The same straightforward computation leads to the critical equality $D\alpha^2 + g'(0) = d\alpha^2 + f'(0)$, or equivalently $\alpha = \sqrt{\frac{f'(0) - g'(0)}{D-d}}$, which separates whether $-\Lambda_g (\alpha)$ is larger or equal to $d\alpha^2 + f'(0)$. Using the fact that $c^*_{KPP} =  \min_{\alpha >0} \frac{d\alpha^2 + f'(0)}{\alpha}$ and that this latter minimum is reached for $\alpha = \sqrt{\frac{f'(0)}{d}}$, one then immediately gets that $c^*_g (D) > c^*_{KPP}$ if and only if 
$$D > 2d - d\frac{g'(0)}{f'(0)}.$$
This extends a result of~\cite{BRR_plus} in the periodic exchange framework. Proposition~\ref{exp_decay_case} also naturally extends and the solutions with exponential decay of order $e^{\alpha x}$ at time $t=0$ are accelerated by the road if and only if $D \alpha^2 + g'(0) > d\alpha^2 + f'(0)$.

\paragraph{The general periodic framework:} In the above extensions and as we already pointed out, we restricted ourselves to problems where only the exchange terms depend on $x$. The main reason is that we can only give an accurate diffusion threshold for the acceleration phenomena in such a framework (see the averaging argument in Section~\ref{sec:accelerate}). On the other hand, the construction of principal eigenvalues in the truncated field mostly relies on the classical Krein-Rutman theory, thanks to the monotone feature of our system. 

This means that this part of our argument, along with the existence of a spreading speed (at least in the truncated problem) could possibly be extended to a much more general framework as long as the $x$-periodicity is maintained. This would include heterogeneous reaction or diffusion (in both $x$ and $y$) and, perhaps more interestingly, curved roads.

Nonetheless, to extend Theorem~\ref{main:speedup} would then be a completely open but interesting problem. It of course seems unlikely that there always exists a simple formula for the diffusion threshold. To exhibit such a threshold, one may need to prove the monotonicity of the principal eigenvalue on some diffusion parameter, which we avoided here. More generally, we hope to be able in a future work to further characterize this eigenvalue and, in a similar fashion as for the single equation~\cite{Nadin10}, derive not only properties on the dependence of the spreading speed on diffusion, but also on various parameters and the exchange terms.

%
%
%
%

\bigskip

\noindent \textbf{Acknowledgement.} T. G. is supported by the
French {\it Agence Nationale de la Recherche} within the project
NONLOCAL (ANR-14-CE25-0013). L. M. was supported by the Portuguese Science Foundation through FCT fellowship SFRH/BPD/ 88207/2012.

The authors are grateful to Henri Berestycki and Jean-Michel Roquejoffre for suggesting this problem and valuable discussions.

\bibliographystyle{siam}
\bibliography{biblio}

\begin{thebibliography}{10}

\bibitem{AW75}
{\sc D.~Aronson and H.~Weinberger}, {\em Nonlinear diffusion in population
  genetics, combustion and nerve propagation}, Partial Diff. Eq. and related
  topics, 446 (1975), pp.~5--49.

\bibitem{AW78}
\leavevmode\vrule height 2pt depth -1.6pt width 23pt, {\em Multidimensional
  nonlinear diffusion arising in population genetics}, Adv. in Math., 30
  (1978), pp.~33--76.

\bibitem{Ber02-adv}
{\sc H.~Berestycki}, {\em The influence of advection on the propagation of
  fronts in reaction-diffusion equations}, Nonlinear PDEs in Condensed Matter
  and Reactive Flows, NATO Science Series~C, 589 (2002).

\bibitem{BCoulonRR}
{\sc H.~Berestycki, A.-C. Coulon, J.-M. Roquejoffre, and L.~Rossi}, {\em
  Speed-up of reaction-diffusion fronts by a line of fast diffusion},
  S\'{e}minaire Laurent Schwartz -- EDP et applications,  (2013-2014).

\bibitem{BH-front}
{\sc H.~Berestycki and F.~Hamel}, {\em Front propagation in periodic excitable
  media}, Comm. Pure Appl. Math., 55 (2002), pp.~949--1032.

\bibitem{BHNn08}
{\sc H.~Berestycki, F.~Hamel, and G.~Nadin}, {\em Asymptotic spreading in
  heterogeneous diffusive media}, J. Funct. Anal., 255 (2008), pp.~2146--2189.

\bibitem{BHNi-I}
{\sc H.~Berestycki, F.~Hamel, and N.~Nadirashvili}, {\em The speed of
  propagation for {KPP} type problems. {I} - {P}eriodic framework}, J. Eur.
  Math. Soc., 7 (2005), pp.~173--213.

\bibitem{BHH-periodic_fragmented}
{\sc H.~Berestycki, F.~Hamel, and L.~Roques}, {\em Analysis of the periodically
  fragmented environment model. {I}. {S}pecies persistence}, J. Math. Biol., 51
  (2005), pp.~75--113.

\bibitem{BHR-Liouville}
{\sc H.~Berestycki, F.~Hamel, and L.~Rossi}, {\em Liouville-type results for
  semilinear elliptic equations in unbounded domains}, Ann. Mat. Pura Appl.,
  186 (2007), pp.~469--507.

\bibitem{BRR_plus}
{\sc H.~{Berestycki}, J.-M. {Roquejoffre}, and L.~{Rossi}}, {\em {Fisher-{KPP}
  propagation in the presence of a line: further effects}}, Nonlinearity, 26
  (2013), pp.~2623--2640.

\bibitem{BRR_influence_of_a_line}
\leavevmode\vrule height 2pt depth -1.6pt width 23pt, {\em {The influence of a
  line with fast diffusion on Fisher-KPP propagation}}, Journal of Mathematical
  Biology, 66 (2013), pp.~743--766.

\bibitem{BRR_shape_of_expansion}
\leavevmode\vrule height 2pt depth -1.6pt width 23pt, {\em The shape of
  expansion induced by a line with fast diffusion in {Fisher-KPP} equations},
  (preprint).

\bibitem{BR-general_eigen}
{\sc H.~Berestycki and L.~Rossi}, {\em Generalizations and properties of the
  principal eigenvalue of elliptic operators in unbounded domains}, Comm. Pure
  Appl. Math.,  (to appear).

\bibitem{Dietrich1}
{\sc L.~Dietrich}, {\em Existence of travelling waves for a reaction-diffusion
  system with a line of fast diffusion},  (preprint).

\bibitem{Dietrich2}
\leavevmode\vrule height 2pt depth -1.6pt width 23pt, {\em Velocity enhancement
  of reaction-diffusion fronts by a line of fast diffusion},  (preprint).

\bibitem{Fisher}
{\sc R.~Fisher}, {\em The wave of advance of advantageous genes}, Annals of
  eugenics, 7 (1937), pp.~355--369.

\bibitem{GilbargTrudinger}
{\sc D.~Gilbarg and N.~S. Trudinger}, {\em Elliptic partial differential
  equations of second order}, Classics in Mathematics, Springer-Verlag, Berlin,
  2001.
\newblock Reprint of the 1998 edition.

\bibitem{Kato-perturbation}
{\sc T.~Kato}, {\em Perturbation theory for linear operators}, Springer-Verlag,
  Berlin, 1995.
\newblock Reprint of the 1980 edition.

\bibitem{KPP}
{\sc A.~Kolmogorov, I.~Petrovskii, and N.~Piskunov}, {\em A study of the
  diffusion equation with increase in the quantity of matter, and its
  application to a biological problem}, Bull. Moscow Univ. Math. Ser. A, 1
  (1937), pp.~1--25.

\bibitem{LadyzhenskayaUraltseva}
{\sc O.~A. Ladyzhenskaya and N.~N. Ural$'$tseva}, {\em Linear and quasilinear
  elliptic equations}, Translated from the Russian by Scripta Technica, Inc.
  Translation editor: Leon Ehrenpreis, Academic Press, New York, 1968.

\bibitem{LiangZhao}
{\sc X.~Liang and X.-Q. Zhao}, {\em Asymptotic speeds of spread and traveling
  waves for monotone semiflows with applications}, Comm. Pure Appl. Math., 60
  (2007), pp.~1--40.

\bibitem{Nadin10}
{\sc G.~Nadin}, {\em The effect of the {S}chwarz rearrangement on the periodic
  principal eigenvalue of a nonsymmetric operator}, SIAM J. on Math. Anal, 41
  (2010), pp.~2388--2406.

\bibitem{Pauthier}
{\sc A.~Pauthier}, {\em The influence of a line with fast diffusion and
  nonlocal exchange terms on {F}isher-{KPP} propagation},  (preprint).

\bibitem{Sattinger-Monotone}
{\sc D.~H. Sattinger}, {\em Monotone methods in nonlinear elliptic and
  parabolic boundary value problems}, Indiana Univ. Math. J., 21 (1971/72),
  pp.~979--1000.

\bibitem{ShiKa}
{\sc N.~Shigesada and K.~Kawasaki}, {\em Biological {I}nvasions: {T}heory and
  {P}ractice}, Oxford University Press, 1997.

\bibitem{Tellini}
{\sc A.~Tellini}, {\em Propagation speed in a strip bounded by a line with
  different diffusion},  (preprint).

\bibitem{Wein02}
{\sc H.~Weinberger}, {\em On spreading speeds and traveling waves for growth
  and migration models in a periodic habitat}, J. Math. Biol., 45 (2002),
  pp.~511--548.

\bibitem{xin00}
{\sc J.~Xin}, {\em Front propagation in heterogeneous media}, SIAM Rev., 42
  (2000), pp.~161--230.

\end{thebibliography}

\end{document}